\documentclass[12pt]{amsart}
\usepackage{amsmath,amscd,amsthm,amsfonts, amssymb,amsxtra, mathrsfs,enumerate,relsize, tikz-cd}
\usepackage{calrsfs}
\usepackage[us,12hr]{datetime}
\usepackage[all]{xy} \SelectTips{cm}{}
\usepackage{hyperref}
 \hypersetup{colorlinks=true,citecolor=blue}
   \usepackage{graphicx}
     \usepackage[margin=1.5in]{geometry}
   \usepackage[font=scriptsize]{caption}
\usepackage{todonotes}
\usepackage{tikz}
\usetikzlibrary{shapes.geometric,decorations.pathreplacing}
\usetikzlibrary{positioning}
\CDat
\usepackage{scalerel}[2016/12/29]

\newtheorem{thm}{Theorem}[section]  
\newtheorem*{un-no-thm}{Theorem}
\newtheorem{cor}[thm]{Corollary}     
\newtheorem{lem}[thm]{Lemma}         
\newtheorem{prop}[thm]{Proposition}

\newtheorem{add}[thm]{Addendum}

\newtheorem{bigthm}{Theorem}

\theoremstyle{definition}
\newtheorem{defn}[thm]{Definition}   

\theoremstyle{definition}
\newtheorem{notation}[thm]{Notation}

\theoremstyle{definition}

\theoremstyle{definition}

\theoremstyle{remark}
\newtheorem{rem}[thm]{Remark}        
\newtheorem{term}[thm]{Terminology}

\newtheorem*{ack}{Acknowledgement}
\newtheorem*{out}{Outline}
\newtheorem{rems}[thm]{Remarks}
 
\newtheorem{ex}[thm]{Example}

\DeclareMathOperator{\Sp}{Spectra}

\DeclareMathOperator{\tr}{tr}
\DeclareMathOperator{\st}{st}

\DeclareMathOperator*{\colim}{colim}
\DeclareMathOperator*{\rel}{rel}
\DeclareMathOperator*{\hocolim}{hocolim}

\begin{document}
\title{Embedding, compression, and the relative Hopf invariant}
\date{\today}
\author{John R.\ Klein}
\address{Wayne State University, Detroit, MI 48202}
\email{klein@math.wayne.edu}
\begin{abstract}  
We establish Poincaré embedding results in the relative setting, generalizing previously known results in the absolute case. 
Our primary motivation comes from applications to non-simply connected Poincaré surgery, which will be developed in a forthcoming paper. 
Along the way, we introduce a new tool: the {\it relative Hopf invariant}.
\end{abstract}
\maketitle
\setlength{\parindent}{15pt}
\setlength{\parskip}{1pt plus 0pt minus 1pt}

\def\Sp{\text{\bf Sp}}
\def\vo{\varOmega}
\def\vs{\varSigma}
\def\smsh{\wedge}
\def\flush{\flushpar}
\def\id{\text{id}}
\def\dbslash{/\!\! /}
\def\codim{\text{\rm codim\,}}
\def\:{\colon}
\def\holim{\text{holim\,}}
\def\Bbb{\mathbb}
\def\bold{\mathbf}
\def\Aut{\text{\rm Aut}}
\def\cal{\mathcal}
\def\sec{\text{\rm sec}}
\def\gda{G\text{\rm -}\delta\text{\rm -}\alpha}
\def\PDD{\text{\rm pd\,}}
\def\PD{\text{\rm P}}
\def\stableto {\,\, \mapstochar \!\!\to}

\setcounter{tocdepth}{1}
\tableofcontents
 
\section{Introduction} \label{sec:intro}
In the 1960s, during the ``golden era'' of manifold theory, a new approach to constructing embeddings of manifolds emerged \cite{Browder_ICM, Haefliger_knotted}, \cite[ch.~11]{Wall}. 
The basic geometric picture is as follows. Let $N \subset M$ be a closed submanifold of codimension at least three. 
A compact tubular neighborhood $U$ of $N$ determines a decomposition
\[
M \cong U \cup_{\partial U} C,
\]
where $C$ is the closure of the complement $M - U$. Thus $(U,\partial U)$ and $(C,\partial U)$ are manifolds with boundary, and
together they exhibit $M$ as a codimension-one splitting.

The surgery approach to manifold classification problems shows that to find a geometric codimension-one splitting, it is often enough to construct an analogous {\it homotopical} codimension-one splitting
by Poincar\'e pairs. More precisely, we seek to identify $M$ as a homotopy pushout of Poincar\'e pairs glued along a common boundary. 
 This is essentially the definition of a \emph{Poincar\'e embedding} of $N$ in $M$. The corresponding homotopy-theoretic problem is to construct such a homotopy pushout.
  
A variant of the homotopy-theoretic
problem is to start with a Poincar\'e space $M$ of dimension $d$ and a Poincar\'e pair $(U,\partial U)$ 
of dimension $d$. 
One then asks if there is another Poincar\'e pair $(C,\partial U)$ such that
the union $U \cup_{\partial U} C$ is homotopy equivalent to $M$. This is known as a {\it codimension-zero Poincar\'e embedding}
of $(U,\partial U)$ in $M$. 

In \cite{Klein_compression} we proved various results about codimension-zero Poincar\'e embeddings. The current article is a sequel: We
provide relative versions of two of the main results of \cite{Klein_compression} and consider two applications.
The idea in this case is to start with Poincar\'e pairs $(U,\partial U)$
and  $(M,\partial M)$ of the same dimension, where a ``portion'' of the boundary $\partial U$ is already embedded inside $\partial M$.
The goal will be to extend the latter to an embedding of $U$ in $M$.
Although the results are interesting in their own right, the main reason  for writing this paper is that relative results are typically more useful 
than absolute ones in applications. In fact,
we will use the results contained here  in another paper
to complete the program of ``Poincar\'e surgery'' in the non-simply connected case.

\subsection{Relative Poincar\'e embeddings} Let 
\[
(P;\partial_0 P,\partial_1 P)
\] be a $d$-dimensional Poincar\'e triad. That is,
$(P,\partial P)$ is a Poincar\'e pair of dimension $d$ equipped with a codimension-one splitting of its boundary as
\[
\partial P = \partial_0 P \cup_{\partial_{01} P} \partial_1 P\, ,
\]
in which $(\partial_0 P, \partial_{01} P)$ and $(\partial_1 P, \partial_{01} P)$ are Poincar\'e pairs of dimension $d-1$.
Note that $\partial_{01} P = \partial_0 P \cap \partial_1 P$. 

\begin{rem} We typically use the notation
\[
(\partial P,\partial_0 P,\partial_1 P)
\]
for a codimension-one splitting. To avoid confusion, we have used a comma rather than
a semicolon after $\partial P$ to distinguish the decomposition from a Poincar\'e triad.
\end{rem}

\begin{defn}[{\cite[p.~764]{GK}}] \label{defn:spine} The {\it homotopy spine dimension} of $(P;\partial_0 P,\partial_1 P)$ (relative to $\partial_0 P$) is $\le p$ 
if
\begin{enumerate}[(i).]
\item the pair $(P,\partial_0 P)$ is homotopically $p$-dimensional (i.e., up to homotopy, $(P,\partial_0 P)$ is a retract of a CW pair of relative dimension $\le p$), and
\item the pair $(P,\partial_1 P)$ is $(d-p-1)$-connected.
\end{enumerate}
\end{defn}

\begin{rem} When $\partial_0 P = \emptyset$, the homotopy spine dimension of $P$ is $\le p$ if and only if
it has homotopy codimension $\ge d-p$ in the sense of \cite[p.~312]{Klein_compression}.
\end{rem}

Suppose that $(M;\partial_0 M,\partial_1 M)$ is another Poincar\'e triad of dimension $d$, in which 
 \[
 (\partial_0 M,\partial_{01} M)  = (\partial_0 P,\partial_{01} P) \, .
 \]
 Then the data amount to
a Poincar\'e embedding of a ``portion'' (i.e., $\partial_0 P$) of $\partial P$ in $\partial M$.

\begin{defn}[{\cite[\S2]{GK}}]
A (codimension-zero)  {\it Poincar\'e embedding} of $P$ in $M$ (relative to $\partial_0 P$) consists of a $d$-dimensional Poincar\'e triad
\[
(C; \partial_1 P, \partial_1 M)
\]
and a homotopy equivalence 
\[
h\: P\cup_{\partial_1 P} C @> \simeq >> M
\]
which restricts to the identity on $\partial M = \partial_0 P \cup_{\partial_{01} P} \partial_1 M$. 

 The Poincar\'e triad $(C;\partial_1 P,\partial_1 M)$  is called the  {\it complement} of the Poincar\'e embedding.
 
 The composition
\[
\begin{tikzcd}
P \ar[r,"\subset"]  & P\cup_{\partial_1 P} C \ar[r,"h", "\sim" '] & M
\end{tikzcd}
\]
is called the {\it underlying map} of the Poincar\'e embedding; it restricts to the identity on $\partial_0 P$.
\end{defn}

\begin{rem} It is important to emphasize that the above definition concerns the notion of a  {\it codimension-zero} Poincar\'e embedding.
For the notion of Poincar\'e embedding in positive codimension, see \cite{Klein_haef}.
\end{rem}

\begin{rem} \label{rem:flex}  In the above definition we assumed $(\partial_0 P,\partial_{01} P) = (\partial_0 M,\partial_{01} M)$. It is sometimes
more convenient to relax the condition by replacing the equality of  pairs with a fixed weak homotopy equivalence:
\[
(\partial_0 P,\partial_{01} P) @> \sim >> (\partial_0 M,\partial_{01} M)\, .
\]
Defining relative embeddings in this way  does not introduce any new technical difficulties, but it sometimes causes notational clutter.
If there is no confusion, we will typically ignore the distinction between the two variants.
\end{rem}

\begin{term} From now on, we use the terms {\it embedding}  and {\it Poincar\'e embedding} interchangeably.
Informally, we declare that a map $f\: P \to M$ (extending the identity
on $\partial_0 P$) {\it embeds} if there is an embedding of $P$ in $M$ (relative to $\partial_0 P$)
with underlying map $f$. When  $(C; \partial_1 P, \partial_1 M)$ is understood,
we sometimes refer to the  embedding by its underlying map $f$.
\end{term}

\begin{figure}
\centering
\begin{tikzpicture}[scale=.8]

\fill[blue!15] (-3,0) rectangle (3,6);
\draw[thick, black] (-3,0) rectangle (3,6);

\draw[thick, blue, domain=-2:2, samples=100] 
  plot (\x, {-(\x)^2 + 4});


\begin{scope}
\clip (-3,0) rectangle (3,6);
\fill[magenta!40, opacity=0.9] 
  plot[domain=2:-2, samples=100] (\x, {-(\x)^2 + 4}) 
  -- cycle;
\end{scope}

\draw[thick, blue] (-2,0) -- (2,0);

\node[font=\fontsize{8pt}{8pt}\selectfont] at (-1.2,3.2) {$\partial_1 \! P$};
\node[font=\fontsize{8pt}{8pt}\selectfont] at (0,.15) {$\partial_0 \! P$};
\node[font=\tiny] at (0,5) {$C$};
\node[font=\tiny] at (0,2) {$P$};
\end{tikzpicture}
\caption{A depiction of a Poincar\'e embedding of $P$ in $M$ relative to $\partial_0 P$, where $\partial_1 M$ is  
the unlabeled part of the boundary.} 
\label{fig:embedding}
\end{figure}
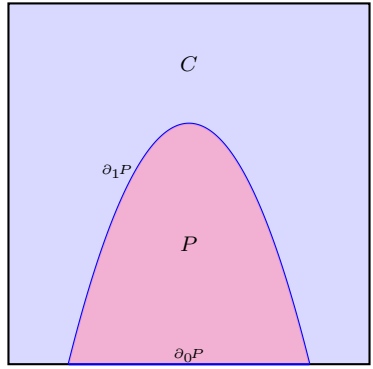

\begin{rem} When $\partial_0 P = \emptyset$, the definition reduces to the one of \cite[\S2.4]{Klein_compression}.
In this instance we call it an {\it interior embedding}.
\end{rem}

\begin{rem} \label{rem:alternative-display} It is sometimes convenient to display a Poincar\'e embedding as a commutative homotopy pushout square of pairs
\[
\xymatrix{
(\partial_1 P, \partial_{01} P) \ar[r] \ar[d] & (C,\partial_1 M) \ar[d] \\
(P,\partial_0 P) \ar[r] & (M,\partial M)\, .
}
\]
\end{rem}

\begin{ex} A codimension-zero smooth embedding $P \to M$ of manifold triads 
determines a Poincar\'e embedding provided that $\partial_1 P$ meets $\partial M$ transversely.
\end{ex}

\subsection{The space of Poincar\'e embeddings}  
Let 
\[
\cal I^h(\partial M)
\] be the category whose objects are Poincar\'e pairs
$(N,\partial N)$ such that $\partial N = \partial M$. A morphism 
$(N,\partial N) \to (N',\partial N')$ is given by a homotopy equivalence $N \to N'$ that restricts to the identity map on  $\partial N = \partial M =  \partial N'$.
It will be convenient in the following to refer to an object $(N,\partial N)$ as $N$.

One has a functor
\[
\kappa^P  \: \cal I^h(\partial_1 P \cup_{\partial_{01} P} \partial_1 M) \to \cal I^h(\partial M)
\]
defined by $C\mapsto C \cup_{\partial_1 P} P$. An object of the over category
$\kappa^P{/M}$ is the same thing as a relative embedding $P\to M$.

\begin{defn} The {\it space of relative embeddings} from $P$ to $M$ is the geometric realization
\[
E(P,M\rel\partial_0 P) = |\kappa^P/M|
\] 
 (cf.~\cite[defn.~2.8]{GK}).  
 
 \begin{rem} The category  $\cal I^h(\partial M)$ has a preferred basepoint object, i.e.,  $M$. 
 The component of the basepoint in the realization $|\cal I^h(\partial M)|$ is a model for  the classifying space
$BG(M\rel \partial M)$, where $G(M\rel \partial M)$ is the topological monoid of homotopy automorphisms of $M$ which restrict
to the identity on $\partial M$
 \cite[prop.~2.2.5]{Wald_1126}.  It is then legitimate to think of
$|\cal I^h(\partial M)|$ as a {\it non-connective delooping} of $G(M\rel \partial M)$ and we write
\[
\cal BG(M\rel \partial M) := |\cal I^h(\partial M)|\, .
\] 
Similarly, a basepoint object of $\kappa^P/M$ has an associated complement 
\[
C \in \cal I^h(\partial_1 P \cup_{\partial_{01} P} \partial_1 M)\, .
\]
Then 
$\cal BG(C\rel \partial C)$ is a non-connective delooping of $G(C\rel \partial C)$.

 By \cite[thm.~B]{Quillen} there is a homotopy fiber sequence
 \[
 E(P,M) \to \cal BG(C\rel \partial C) \to \cal BG(M\rel \partial M)\, .
 \]
 If we shift the sequence back two steps, we obtain  another homotopy fiber sequence 
 \[
 G(C\rel \partial C) \to G(M\rel \partial M) \to  E(P,M)\, .
 \]
The map $G(M\rel \partial M) \to  E(P,M)$ is given by the orbit of the basepoint of $E(P,M)$
 with respect to the action $G(M\rel \partial M) \times E(P,M) \to E(P,M)$ which takes an embedding
 and composes it with an automorphism of $M$. We omit the details.
 \end{rem}

A {\it concordance} is a 1-simplex of $\kappa^P/M$.
Two embeddings of $P$ in $M$ relative to $\partial_0 P$
are said to be {\it concordant} if they lie in the same path component of $E(P,M \rel \partial_0 P)$.
\end{defn}

\begin{rem} A pair of Poincar\'e embeddings of $f_0,f_1\: P\to M$ assemble in the obvious way to define
a Poincar\'e embedding
\[
\partial_0(P \times I) := P\times \{0\} \cup \partial_0 P \times I  \cup P\times \{1\} \to \partial(M \times I)\, .
\]
Then a Poincar\'e embedding of $P \times I$ in $M \times I$ which is fixed on $\partial_0(P \times I) $ (in the sense of Remark \ref{rem:flex}) is the same thing as a concordance.
\end{rem}

\begin{notation} To avoid clutter, we will typically drop $\partial_0 P$ from the notation and write
\[
E(P,M) := E(P,M\rel\partial_0 P) \, .
\]
\end{notation}

Analogously, let $F(P,M) := F(P,M \rel \partial_0 P)$ denote the function space of maps $P \to M$ which 
are fixed on $\partial_0 P$. The operation which assigns the underlying map to a Poincar\'e embedding
yields a map of spaces
\[
E(P,M) \to F(P,M)
\]
(cf.~\cite[\S2]{GK}).

\subsection{The Compression Theorem} 
For a space $X$, let $T_{/X}$ be the 
category of spaces over $X$ and let $R(X)$ be the 
category of retractive spaces over $X$. Both $T_{/X}$ and $R(X)$ have the structure of Quillen model categories (cf.~\S\ref{sec:prelim}), so
the notions of fibration, cofibration, and weak equivalence are defined. Therefore the homotopy categories $\text{ho}T_{/X}$ and $\text{ho}R(X)$
are defined.

If  $A\to B$ is a cofibration of $T_{/X}$,  then  the (fiberwise) quotient is  the object
\[
B/\!\!/ A := B \cup_A X \in R(X)\, .
\]
As a special case, we have $B^+ := B/\!\!/\emptyset = B \amalg X$. 
For objects $Y,Z \in R(X)$ we let
\[
[Y,Z]_{R(X)}
\]
be the set of fiberwise homotopy classes (cf.~ Definition \ref{defn:function-space}  below).

Let 
\[
(P;\partial_0 P,\partial_1 P) \qquad \text{ and } \qquad (M;\partial_0 M,\partial_1 M) 
\]
be $d$-dimensional Poincar\'e
triads as above, with $(\partial_0 P,\partial_{01} P) = (\partial_0 M,\partial_{01} M)$.

Let $I = [0,1]$ be the unit interval and let $J = [1/3, 2/3]$. 
Then the inclusion $j\: J @>>> I$ induces an interior  embedding
\[
\partial_0 P \times J \subset \partial_0 P \times I  \subset \partial (M \times I)
\]
defining a codimension-one splitting
\[
(\partial (M \times I), \partial_0 P \times J,C)\, .
\]
Moreover, if there is a relative embedding $f\: P \to M$, then
we have an induced embedding $f\times j\: P \times J \to M \times I$
relative to $\partial_0 P \times J$. In this case, one says that $f\times j$ is the {\it decompression} of
$f$.

Conversely, suppose one is provided with an embedding $P\times J \to M\times I$ relative to 
$\partial_0 P \times J$, i.e., there is a codimension-one splitting of $P \times J$ in $M\times I$:
\[
(M\times I, P\times J,W)
\]
with $(P \times J) \cap \partial (M \times I) =  \partial_0 P \times J$ and $W \cap \partial (M \times I) = C$ (cf.~Figure \ref{fig:to-be-compressed}).
We would like to decide when the latter is concordant to a decompression of a relative  embedding $P\to M$.

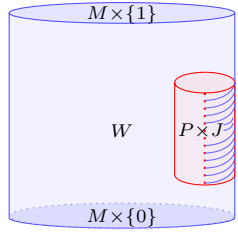
\begin{figure}
\centering
\begin{tikzpicture}[scale=1]

\node[
  cylinder,
  draw=blue!70,
  minimum width=3cm,
  minimum height=3cm,
  shape border rotate=90,
  cylinder uses custom fill,
  cylinder body fill=blue!20,
  cylinder end fill=blue!40,
  fill opacity=0.25
] (big) at (0,0) {};

\draw[dotted,
  blue!60,
  fill=blue!30,
  fill opacity=0.35
]
  ([yshift=0.17cm]big.south) ellipse[x radius=1.5cm, y radius=0.17cm];

\node[
  cylinder,
  draw=red,
  minimum width=0.8cm,
  minimum height=1.5cm,
  shape border rotate=90,
  cylinder uses custom fill,
  cylinder body fill=red!20,
  cylinder end fill=red!30,
  fill opacity=0.2
] (small) at (1.1,-0.2) {};

\draw[red, dotted, line width=.8pt]
  ([yshift=0.60cm]small.center) -- ([yshift=-0.65cm]small.center);

\node[font=\relsize{-3}\smaller] at (1.05,-.1) {$P {\times} J$};
\node[font=\relsize{-3}\smaller] at (0,-.1) {$W$};
\node[font=\relsize{-3}\smaller] at (0,1.45) {$M {\times} \{1\}$};
\node[font=\relsize{-3}\smaller] at (0,-1.25) {$M {\times} \{0\}$};

      \draw[blue!60] (1.1,.3) arc[start angle=-90, end angle=0, x radius=.4cm, y radius=.25cm];
    \draw[blue!60] (1.1,.2) arc[start angle=-90, end angle=0, x radius=.4cm, y radius=.25cm];
  \draw[blue!60] (1.1,.1) arc[start angle=-90, end angle=0, x radius=.4cm, y radius=.25cm];
  \draw[blue!60] (1.1,0) arc[start angle=-90, end angle=0, x radius=.4cm, y radius=.25cm];
        \draw[blue!60] (1.35,-.025) arc[start angle=-80, end angle=0, x radius=.15cm, y radius=.1cm];
    \draw[blue!60] (1.35,-.1) arc[start angle=-80, end angle=0, x radius=.15cm, y radius=.1cm];
    \draw[blue!60] (1.1,-.3) arc[start angle=-90, end angle=0, x radius=.4cm, y radius=.25cm];
    \draw[blue!60] (1.1,-.4) arc[start angle=-90, end angle=0, x radius=.4cm, y radius=.25cm];
        \draw[blue!60] (1.1,-.5) arc[start angle=-90, end angle=0, x radius=.4cm, y radius=.25cm];
            \draw[blue!60] (1.1,-.6) arc[start angle=-90, end angle=0, x radius=.4cm, y radius=.25cm];
                  \draw[blue!60] (1.1,-.7) arc[start angle=-90, end angle=0, x radius=.4cm, y radius=.25cm];
                        \draw[blue!60] (1.1,-.8) arc[start angle=-90, end angle=0, x radius=.4cm, y radius=.25cm];
\end{tikzpicture}
\caption{Depiction of a relative Poincar\'e embedding of $P\times J$ in $M \times I$. The ruled portion of $P\times J$ 
represents the embedding of $\partial_0 P\times J$ in $\partial (M\times I)$.}
\label{fig:to-be-compressed}
\end{figure}

Consider the evident map of pairs
\[
(P \times \{1/3\},\partial_0 P \times \{1/3\}) \to (W,C)\, .
\]
Then the map $\partial_0 P \times \{1/3\} \to C$ has a preferred homotopy (over $\partial M$) to a factorization of the form
\[
\partial_0 P \times \{1/3\} @> \subset >> \partial M \times \{0\} \subset C\, ,
\]
where the first map is the inclusion (recall that $\partial_0 P = \partial_0 M$).

Then the map of pairs
determines a fiberwise homotopy class 
\[
\lambda \in [P/\!\! /\partial_0 P, W]_{R(M)}\, .
\]
(Note that $W$ is an object of $R(M)$ since there is a factorization $M\times\{0\} \to W \to M \times I$.)

\begin{defn} The homotopy class $\lambda$ is called the {\it link} of the
embedding $P\times J \to M \times I$.
\end{defn}

Note that if the embedding compresses, then $\lambda$ is trivial.

\begin{bigthm}[Compression Theorem] \label{bigthm:compression} Assume that the homotopy  spine dimension of $P$ is $\le p$, with $p \le d-3$ and $3p+4\le 2d$.
If $\lambda$ is trivial, then up to concordance, the relative embedding $P\times J \to M \times I$ compresses to a relative embedding $P\to M$.
\end{bigthm}

\begin{rem} When $\partial_0 P = \emptyset$, Theorem \ref{bigthm:compression} yields \cite[thm.~A]{Klein_compression}.
\end{rem}

\begin{rem} 
With respect to the numerical assumptions of Theorem \ref{bigthm:compression},
the link $\lambda$  is {\it stable} in the sense that
the function
\[
E\: [P/\!\! /\partial_0 P, W]_{R(M)} \to \{P/\!\! /\partial_0 P, W\}_{R(M)}
\] 
is a bijection, where the target of $E$ is the abelian group of fiberwise stable homotopy classes
of morphisms $P/\!\! /\partial_0 P\to W$ \cite[prop.~9.3]{James_fibrewise}
 (compare ~\cite[add.~1.3]{Klein_compression}). 
\end{rem}

\subsection{Poincar\'e immersions} 
 Let 
 \[
 (P;\partial_0 P,\partial_1 P)  \qquad \text{ and } \qquad (M;\partial_0 M,\partial_1 M)
 \]  be as above, with
 $(\partial_0 P,\partial_{01} P) = (\partial_0 M,\partial_{01} M)$.  Suppose $f\:P \to M$ is a relative embedding.

Let $j \ge 0$ be an integer. Then one has a Poincar\'e triad of dimension $d+j$:
\[
(P\times D^j; (\partial_0 P) \times D^j, S_P^j\partial_1 P)\, , 
\]
where $S_P^j\partial_1 P = ((\partial_1 P) \times D^j) \cup_{(\partial_1 P) \times S^{j-1}} (P \times S^{j-1})$
is the $j$-fold unreduced fiberwise suspension of $\partial_1 P$ over $P$. Here we have implemented the Poincar\'e space decomposition
\begin{align*}
\partial (P\times D^j) &= ((\partial P) \times D^j)\cup (P \times S^{j-1})\, ,\\
&= ((\partial_0 P) \times D^j) \cup ((\partial_1 P) \times D^j) \cup (P \times S^{j-1})\, ,\\
&= ((\partial_0 P) \times D^j) \cup S_P^j\partial_1 P\, .
\end{align*}

\begin{defn}[{cf.~ \cite[\S1.1]{Klein_compression}}]  \label{defn:immersion} Let $f\: P \to M$ be a map
fixed on $\partial_0 P$. A {\it Poincar\'e immersion} of $f$ is an embedding
of $P \times D^j \to M \times D^j$ relative to $(\partial_0 P) \times D^j$ with underlying map
\[
f\times \text{id}\: P \times D^j \to M \times D^j
\]
for some $j \ge 0$.

In particular, an embedding $P\to M$ is also an immersion. One says that an immersion of $P$ in $M$  {\it compresses}
 if, up to concordance, it arises from a relative embedding of $P$ in $M$.
 \end{defn}

\begin{rem}[Relation to Smale-Hirsch theory] Definition \ref{defn:immersion}  was motivated by the following considerations: Let $P\to M$ be an immersion of smooth $n$-manifolds
with $P$ compact and having spine dimension $p \le n-3$. By transversality, there is a $j > 0$ such that the composition 
$P \to M \subset M \times D^j$ is regularly homotopic to an embedding.
The normal bundle of this embedding is trivial, so a tubular neighborhood defines an embedding $P \times D^j \to M \times D^j$. Conversely, by Smale-Hirsch theory,
up to isotopy any embedding  $P \times D^j \to M \times D^j$ arises from an immersion $P\to M$.
\end{rem}

Let $\nu_P$ denote the Spivak normal fibration of $(P,\partial P)$ and let $\nu_M$ denote the Spivak normal fibration of $(M,\partial M)$
(cf.~\S\ref{subsubsec:Spivak} below for the definition).

The following result is a relative version of  \cite[thm.~A]{Klein_immersion}. 

\begin{bigthm} \label{bigthm:Smale-Hirsch}  Let $f\: P\to M$ be a map that restricts to the identity on $\partial_0 P$. Then 
$f$ is the underlying map of a Poincar\'e immersion if and only if there is a stable fiber homotopy equivalence $\nu_P \simeq f^\ast \nu_M $
which restricts to the identity on $\partial_0 P$.
\end{bigthm}

\subsubsection{The space of immersions} 
If $P \to M$ is an embedding with complement data $(C; \partial_1 P,\partial_1 M)$, then the  
operation
\[
(C; \partial_1 P,\partial_1 M) \mapsto (S_M C; S_P\partial_1 P, S_M \partial_1 M)
\]
induces a {\it decompression} (or {\it stabilization})  map 
\[
E(P,M) \to E(P \times D^1,M \times D^1) \, 
\]
(by convention, the target of the decompression map is
 $E(P \times D^1,M \times D^1) := E(P \times D^1,M \times D^1\rel \partial_0 P \times D^1)$).
 In the following
we identify $D^j$ with $(D^1)^{\times j}$.
 
\begin{defn} \label{defn:immersion-space}
The {\it space of immersions} 
(relative to $\partial_0 P$)
is  the homotopy colimit
\[
I(P,M) := \hocolim_{j\to \infty} E(P \times D^j,M \times D^j)\, .
\]
\end{defn}

Similarly, let $F(P \times D^j, M\times D^j)$ be the space of maps $P\times D^j \to M \times D^j$ relative to $\partial_0 P \times D^j$.
Then the product with $D^j$ defines a weak equivalence of function spaces 
\[
F(P,M) \simeq F(P \times D^j, M\times D^j)\, .
\]
Consequently,  we have a well-defined map
\[
I(P,M) \to F(P,M)
\]
up to natural equivalence which assigns to an immersion its underlying map.

\subsubsection{The Pontryagin-Thom collapse}
Associated with an embedding $f\: P \to M$ with complement $C$ there is 
a (fiberwise) {\it collapse map} of pairs
\begin{equation*}
\mbox{$
 (M/\!\!/\partial_1 M,\partial_0 M/\!\!/\partial_{01} M) @>>> (M/\!\!/C,\partial M/\!\!/\partial_1 M) @< \sim << (P/\!\!/\partial_1 P,\partial_0 P/\!\!/\partial_{01} P)\, ,$}
\end{equation*}
where the restriction to $\partial_0 M/\!\!/\partial_{01} M$
is the identity. 
Consequently, the collapse map amounts to a  fiberwise homotopy class 
\[
f^! \in [M/\!\!/\partial_1 M,P/\!\!/\partial_1 P; \rel \partial_0 M/\!\!/\partial_{01} M]_{R(M)}\, .
\]
Similarly, an immersion $f\: P \to M$ induces a (fiberwise)  {\it stable collapse}
\[
f^! \in \{M/\!\!/\partial_1 M,P/\!\!/\partial_1 P \rel  \partial_0 M/\!\!/\partial_{01} M\}_{R(M)}\, .
\]

\begin{bigthm}\label{bigthm:unstable-normal} Suppose  that the homotopy spine dimension of $P$ is $\le p$ with $p \le d-3$ and $3p+4\le 2d$. Then
 up to concordance an immersion $f\: P \to M$ compresses to an embedding
if and only if $f^!$ lies in the image of the stabilization map 
\begin{equation*}
\mbox{\small $E\: [M/\!\!/\partial_1 M,P/\!\!/\partial_1 P \rel  \partial_0 M/\!\!/\partial_{01} M]_{R(M)} \to \{M/\!\!/\partial_1 M,P/\!\!/\partial_1 P \rel  \partial_0 M/\!\!/\partial_{01} M\}_{R(M)}.$}
\end{equation*}
Furthermore, given an unstable class $\alpha$ with $E(\alpha) = f^!$, the embedding can be chosen
 so that its unstable collapse is $\alpha$.
\end{bigthm}

For an alternative formulation of Theorem \ref{bigthm:unstable-normal} see Remark \ref{rem:unstable-normal}.

\subsection{Applications} We provide two applications of the above results. 
The  first of these is a relative version of  \cite[cor.~C]{Klein_compression}. It roughly says that if the homotopical codimension exceeds the homotopy
spine dimension $p$, then every relative immersion compresses.

\begin{bigthm}[{Compare \cite[cor.~C]{Klein_compression}}] \label{bigthm:Whitney} Assume that the homotopy spine dimension of $P$ is $\le p$ 
with $d \ge \max(p+3,2p+1,\frac{3}{2}p + 2)$. Then up to concordance, a relative immersion $f\: P \to M$ compresses to  an  embedding. 
\end{bigthm}

 \begin{rem} When $d \ge 7$, the inequality $d \ge \max (p+3,2p+1,\frac{3}{2}p + 2)$ simplifies to  $d \ge 2p+1$.
 \end{rem}

The next result concerns embeddings of handles in the middle dimension of an even-dimensional Poincar\'e pair. The result is
crucial to the construction of handlebody structures on Poincar\'e spaces, a construction that will be developed elsewhere.

Suppose that $(M,\partial M)$ is a Poincar\'e pair of dimension $d= 2k$, where $M$ is connected, $\partial M$ is based,
and $(M,\partial M)$ is $2$-connected.
Consider the $k$-handle
\[
(H,\partial_0 H,\partial_1 H) := (D^k\times D^{k},S^{k-1} \times D^{k},D^k \times S^{k-1})\, .
\]
Let 
\[
[\alpha]\in \pi_k(M,\partial M)
\]
be a homotopy class.

\begin{bigthm}\label{bigthm:handle} Assume in addition that $k \ge 4$. Then there is a Poincar\'e triad 
\[
(N;\partial_0 N,\partial_1 N)
\]
with $\partial_0 H = \partial_0 N$ and
 a weak homotopy equivalence $(N,\partial N) @> {}_\sim >> (M,\partial M)$.

Furthermore, with respect to the isomorphism $\pi_k(N,\partial N) \cong \pi_k(M,\partial M)$, 
the homotopy class $[\alpha]$ is represented by an embedding
\[
H \to N
\]
relative to $\partial_0 H$.
\end{bigthm}


\begin{rems} \phantom{dfrrvs}
\smallskip

\noindent (1). Informally, Theorem \ref{bigthm:handle} says that $M$ admits a 
Poincar\'e triad structure up to weak equivalence such that if we identify $M$ with the triad, we have
$\partial_0 M = \partial_0 H$. Furthermore, 
 $[\alpha]$ is represented
by an embedding $H \to M$ where the restriction to  $\partial_0 H$ is an embedding 
of the form
 \[
 \partial_0 H \to  \partial M \, .
 \]
In the theorem,  one does not {\it a priori}
fix an embedding $\partial_0 H \to \partial M$ in advance. In other words,
it is not clear whether a given embedding $\partial_0 H \to \partial M$ can be extended
to a relative embedding $H \to M$.
\medskip

\noindent (2). A related statement appears in \cite[4.19(b)]{HV}, where the proof
uses manifold techniques and several nested double inductions. The proof given here is homotopy-theoretic.
\end{rems}

\subsection{The relative Hopf invariant}
The proof of Theorem \ref{bigthm:handle} involves the introduction of a  {\it relative Hopf invariant}, which
is an operation on stable homotopy classes of equivariant  maps.
In the unequivariant case, it is a natural transformation of the form
\[
H_A\: \{X,Y;A\} \to \{X/A,D_2(Y)\}\, ,
\]
where the domain is the set of homotopy classes of stable maps $X\to Y$ which are unstable when restricted to $A$.
The main result about $H_A$ used in the paper is that it provides a criterion for when a stable map which is unstable when restricted to $A$
desuspends relatively to an unstable map in the metastable range; see Lemma \ref{lem:fringe} below. 

\begin{out} Section \ref{sec:prelim} introduces the model categories of spaces that appear in the paper.
In Section \ref{sec:Hopf-invariant} we introduce the (equivariant stable) Hopf invariant. After listing its properties, we derive
the Action Lemma. Although the latter is not used in the paper, it serves as motivation for the Gluing Lemma
in  Section \ref{sec:rel-Hopf}.
Section \ref{sec:rel-Hopf} introduces the relative Hopf invariant, proves the Gluing Lemma and the Fringe Lemma; these are used in the proof
of Theorem \ref{bigthm:handle}. Section \ref{sec:compression}
contains the proof of Theorem \ref{bigthm:compression}. Section \ref{sec:immersions} is about the proof of Theorem \ref{bigthm:Smale-Hirsch}. In Section \ref{sec:unstable-normal} we prove
Theorems \ref{bigthm:unstable-normal} and \ref{bigthm:Whitney}. Section \ref{sec:handle} contains the proof
of Theorem  \ref{bigthm:handle}. In Appendix \ref{sec:thickenings} we develop the notion of a relative Poincar\'e thickening; this is the 
principal tool used in the proof of Theorem \ref{bigthm:Smale-Hirsch}.
\end{out}

\begin{ack} I am grateful to Shmuel Weinberger for his steady encouragement in the current endeavor. 
I am also indebted to the referee for a careful reading of the paper, which led to substantial improvements in the exposition.
\end{ack}

\section{Preliminaries} \label{sec:prelim}
Let $\cal C$ be a Quillen model category. 
Then the homotopy category $\text{ho}\cal C$ is defined and it has the same objects as $\cal C$. For objects
$X,Y\in \cal C$, we write
\[
[X,Y]_{\cal C}  = \hom_{\text{ho} \cal C}(X,Y)
\]
for the set of homotopy classes.

If $\phi\: A \to B$ is a map of $\cal C$, then
we let 
\[
\cal C_\phi 
\] be the {\it category of factorizations} of $\phi$. An object of $\cal C_\phi$
is a  triple $(X,r,s)$ in which $X\in \cal C$ is an object, $s\: A\to X$ and $r\: X\to B$ are maps, and $r\circ s = \phi$. 
A morphism $(X,r,s) \to (X',r',s')$ of $\cal C_\phi$ is a map $f\: X\to X'$ 
of $\cal C$ such that $f\circ s = s'$ and $r'\circ f = r$.  When $r$ and $s$ are understood, we write
$X$ in place of $(X,r,s)$. There is a forgetful functor $\cal C_\phi \to \cal C$.

Declare a map $f\: X\to Y$ of $\cal C_\phi$ to be  a weak equivalence (fibration, cofibration) if and only if
it is a weak equivalence (resp.~fibration, cofibration) when considered in $\cal C$.
In particular, an object $(X,r,s)$ is cofibrant if and only if
$s\: A\to X$ is a cofibration of $\cal C$ and $(X,r,s)$ is fibrant if and only if
$r\: X\to B$ is a fibration.

\begin{lem} \label{lem:factorizations} With respect to the above, $\cal C_\phi$ has the structure of a model category.
\end{lem}

\begin{proof} The category $\cal C_\phi$ coincides with the category
\[
(A,\phi)\backslash(\cal C/B)\, .
\]
The result then follows by application of \cite[prop.~2.2.6]{Quillen-homotopical}.
\end{proof}

%

%


\subsection{Model categories of spaces} Let $T$ be the model category of
compactly generated weak Hausdorff spaces \cite[thm.~2.4.25]{Hovey}.
If $\text{map}(X,Y)$ denotes the space of maps in the compact-open topology,
then 
\[
F(X,Y) := k\text{map}(X,Y)
\]
is the space obtained from $\text{map}(X,Y)$ by retopologizing it in the compactly generated topology.
Then $T$ is a topologically enriched model category.

A non-empty space $X$ 
is {\it $0$-connected} if it is path connected.
It is $r$-connected with  $r > 0$ if it is $0$-connected and the set of homotopy classes
$[S^k,X]$ has one element for $k\le r$.
A map of non-empty (possibly unbased) spaces is
{\it $r$-connected} if its homotopy fiber, taken with respect to
all choices of basepoint, is $(r{-}1)$-connected. 

Let $\pi$ be a group. Let $T(\pi)$ be the category of spaces with left $\pi$-action,
i.e., an object of $T(\pi)$ is a space $X$ equipped with an action $\pi \times X \to X$.
A morphism of $T(\pi)$ is an equivariant map. There is a forgetful functor
$T(\pi) \to T$. Declare a morphism $f\: X\to Y$ of $T(\pi)$ to be a weak equivalence (resp.~fibration)
if and only if it is a weak equivalence (resp.~fibration) of $T$. Declare $f$ to be a cofibration if and only if
$f$ possesses the left lifting property with respect to the acyclic fibrations.

\begin{lem} \label{lem:pi-space} With respect to this structure, $T(\pi)$ is a model category.
\end{lem}

\begin{proof} Let $\pi$ be considered as a category with one object $\ast$ and
$\hom(\ast,\ast) = \pi$. Then $T(\pi)$ is the category of diagrams $T^\pi$, i.e., the category
of functors $\pi \to T$.
The result then follows by application of \cite[thm.~11.6.1]{Hirschhorn}.
\end{proof}

\begin{rem} Let $X\in T(\pi)$ be an object and suppose that
$f\: S^{j-1} \times \pi \to  X$ is an equivariant map. Then the space
$X\cup_f (D^j \times \pi)$ is the result of attaching a free $j$-cell to $X$.
If $Y$ is obtained from $X$ by attaching arbitrarily many free cells, then 
the inclusion $X\to Y$ is a cofibration. 
Conversely, every cofibration of $T(\pi)$ arises as a retract of such an inclusion $X\to Y$.
\end{rem}

Let $\phi\: A\to B$ be an equivariant map of left $\pi$-spaces. Then the category of factorizations
\[
T(\pi)_{\phi}
\]
is a model category by Lemmas \ref{lem:pi-space} and \ref{lem:factorizations}.

Here are some specific cases of interest:
\begin{enumerate}[(i).]
\item The group $\pi$ is trivial and $\phi\: \emptyset \to X$. In this
case $T(\pi)_\phi$ is $T_{/X}$,  the model category of spaces over $X$.
\item The group $\pi$ is trivial and $\phi\: X\to X$ is the identity map. In this case we denote the category $T(\pi)_\phi$ by
\[
R(X)\, .
\]
This is the model category of retractive spaces over $X$.
\item The group $\pi$ is arbitrary and $\phi\: X\to X$ is the identity map. Then $T(\pi)_\phi$
is denoted by
\[
R(X,\pi)\, .
\]
This is the model category of equivariant retractive spaces over $X$.

\item The group $\pi$ is arbitrary and $\phi\: \ast \to \ast$ is the identity map of a point. In this case, we denote the category by 
\[
T_\ast(\pi)\, .
\]
This is the model category of based $\pi$-spaces.
\item The group $\pi$ is arbitrary and $\phi\: A \to \ast$ is the constant map to a point. In this case, we denote the category by 
 \[
 T_A(\pi)\,  .
 \]
This is the model category of  $\pi$-spaces under $A$.
\end{enumerate}

\begin{defn}[Connectivity and Dimension] \label{defn:conn-dim} Let $\phi\: A \to B$ be a morphism of $T(\pi)$ (or ~$T_\ast(\pi)$).
\begin{enumerate}[(i).]
\item A morphism $X\to Y$ of $T(\pi)_\phi$ is {\it $r$-connected} if and only if it is $r$-connected when considered in $T$.
\item An object $X \in T(\pi)_\phi$ is {\it $r$-connected} if the unique map to the terminal object, i.e.,  the map $X\to B$, is $(r+1)$-connected. 
\item We write $\dim_A X \le n$ if $X$ is the retract of an object built from $A$ by attaching free cells of dimension at most $n$; here we note that
if $Z$ is an object of $T(\pi)_\phi$ then so is $Z \cup (D^n \times \pi)$.
\end{enumerate}
\end{defn}

\subsubsection{Suspension} \label{subsubsec:suspension} Let $Y \to X$ be a map of $\pi$-spaces. Then the {\it unreduced fiberwise suspension} is the double mapping cylinder
\[
S_X Y = X\times \{0\} \cup Y \times I \cup X\times \{1\}\, .
\]
If $Y \in R(X,\pi)$ is an object, then $S_X(Y)$ comes equipped with a map $S_X X \to S_X Y$, where $S_X X = X\times I$.
The space
\[
\Sigma_X Y := S_X Y \cup_{X\times I} X 
\]
obtained from $S_XY$ by collapsing the subspace $X\times I \subset S_XY$ to $X$
is the {\it reduced fiberwise suspension}; the assignment $Y \mapsto \Sigma_XY$ defines an  endo-functor $R(X,\pi) \to R(X,\pi)$.

\subsubsection{Function spaces} \label{subsubsec:function-space} Let $\cal C = R(X,\pi)$.

For objects $Y,Z \in \cal C$, we let
\[
F_{\cal C}(Y,Z)
\]
denote the space of morphisms $Y\to Z$, where the topology is obtained by 
considering it as a subspace of $F(Y,Z)$.

The {\it derived function space} is 
\[
\cal F_{\cal C}(Y,Z) = F_{\cal C}(Y^{\text{c}}, Z^{\text{f}})\, ,
\]
where  $Y^{\text{c}}$ is a cofibrant approximation to $Y$ and
$Z^{\text{f}}$ is a fibrant approximation to $Z$. 

The {\it derived function space of stable maps} $Y\to Z$ is
\[
\cal F^{\st}_{\cal C}(Y,Z) = \colim_{j} \cal F_{\cal C}(\Sigma_X^j Y,\Sigma_X^j Z)\, ,
\]
where the maps of the colimit are given by application of $\Sigma_X$.

There is a stabilization map $\cal F_{\cal C}(Y,Z) \to \cal F^{\st}_{\cal C}(Y,Z)$.

More generally, suppose that a cofibration $A\to Y$ is given. Define
\[
\cal F^{\st}_{\cal C}(Y,Z;A) := \lim_{\leftarrow} (\cal F^{\st}_{\cal C}(Y,Z) \to  \cal F^{\st}_{\cal C}(A,Z) \leftarrow   \cal F_{\cal C}(A,Z))  \, .
\]
The latter is the derived function space of stable maps $Y \to Z$ which are unstable when restricted to $A$.
Here, the displayed limit coincides with the corresponding homotopy limit, since the restriction map in the diagram is a fibration.

Note that the stabilization map factors as
\[
\cal F_{\cal C}(Y,Z) \to \cal F^{\st}_{\cal C}(Y,Z;A) \to \cal F^{\st}_{\cal C}(Y,Z)\, . 
\]
Suppose in addition that $A\to Z$ is a fixed map. Define
\[
\cal F_{\cal C}(Y,Z\rel A) : = \text{fiber}(\cal F_{\cal C}(Y,Z) \to \cal F_{\cal C}(A,Z))\, ,
\]
where the fiber is taken at the evident basepoint. Similarly, one has $\cal F^{\st}_{\cal C}(Y,Z\rel A)$.

\begin{rem} If $Y$ is cofibrant and $Z$ is fibrant, we
have an  identification 
\[
\cal F_{\cal C}(Y,Z) \simeq F_{\cal C}(Y,Z)\, .
\]
\end{rem}

\begin{defn}[Homotopy Classes] \label{defn:function-space}
  Let  $Y,Z\in \cal C$ be objects, let $A\to Y$ be a cofibration and let $A\to Z$ be a fixed map.
  We define various sets of homotopy classes as follows:
  \begin{enumerate}[(i).]
\item $[Y,Z]_{\cal C} := \pi_0(\cal F_{\cal C}(Y,Z))$;
\item $[Y,Z\rel A]_{\cal C} := \pi_0(\cal F_{\cal C}(Y,Z\rel A))$;
\item $\{Y,Z\}_{\cal C} := \pi_0(\cal F^{\st}_{\cal C}(Y,Z))$;
\item $\{Y,Z; A\}_{\cal C}  := \pi_0(\cal F^{\st}_{\cal C}(Y,Z;A))$;
\item $\{Y,Z\rel A\}_{\cal C}  := \pi_0(\cal F^{\st}_{\cal C}(Y,Z \rel A))$.
\end{enumerate}
\end{defn}

Note that the stabilization maps determine  functions
\[
[Y,Z]_{\cal C}  \to  \{Y,Z\}_{\cal C}\, , \quad [Y,Z]_{\cal C}  \to  \{Y,Z; A\}_{\cal C}\, , \quad \text{and } \quad [Y,Z\rel A]_{\cal C}  \to  \{Y,Z\rel A\}_{\cal C}
\]
which in each case we denote by $E$.

\begin{rem}[Reformulation of Theorem \ref{bigthm:unstable-normal}] \label{rem:unstable-normal} 
 With respect to the assumptions of Theorem \ref{bigthm:unstable-normal}, a useful restatement of the result is that
the square
\[
\xymatrix{
E_f(P,M) \ar[r] \ar[d] &   \cal F_{R(M)}(M/\!\!/\partial_1 M,P/\!\!/\partial_1 P \rel  \partial_0 M/\!\!/\partial_{01} M) \ar[d]\\
I_f(P,M) \ar[r] & \cal F_{R(M)}^{\st}(M/\!\!/\partial_1 M,P/\!\!/\partial_1 P \rel  \partial_0 M/\!\!/\partial_{01} M)
}
\]
is $0$-Cartesian, where $E_f(P,M)$ is the homotopy fiber of the map $E(P,M) \to F(P,M)$ at $f \in F(P,M)$
and $I_f(P,M)$ is defined analogously.
The horizontal maps of the square  are given by the Pontryagin-Thom collapse.
\end{rem}

\subsection{Poincar\'e duality spaces}
A space $P$ is a {\it Poincar\'e duality space} 
of  dimension $d$ if there exists a pair
\[
(L,[P])
\]
in which $L$ is a  bundle  of coefficients that is locally isomorphic
to $\Bbb Z$, and  $[P] \in H_d(P;L )$ is a homology class
such that the associated cap product homomorphism
\[
\cap [P]\:H^*(P;\cal B) \to H_{d{-}*}(P;\cal B \otimes L )
\]
is an isomorphism in all degrees for every local coefficient bundle $\cal B$ \cite{Wall_PD}.
If the pair $(L,[P])$ exists, then it
is defined up to unique isomorphism; $L$ is called
the {\it orientation sheaf} and $[P]$ the {\it fundamental class}.

Similarly, given a Poincar\'e space $\partial P$
of dimension $d-1$, one has the notion of a {\it Poincar\'e pair} $(P,\partial P)$ of dimension $d$,
where now $[P] \in H_d(P,\partial P;L )$ induces an isomorphism
\[
\cap [P]\:H^*(P;\cal B) \to H_{d{-}*}(P,\partial P; \cal B\otimes L)\, 
\]
(cf.~\cite[thm.~B]{Klein-Qin-Su}).
In this paper, we will assume that Poincar\'e spaces (pairs) are cofibrant
and finitely dominated, i.e., they are a retract of a finite CW complex
(resp.~finite CW pair) up to homotopy.

\subsection{Stable spherical fibrations}
Let $BG$ be the classifying space of stable spherical fibrations \cite{Stasheff_BG}. 
This is the classifying space of the grouplike topological monoid $G$ consisting of
stable self-equivalences of the zero sphere $S^0$.

By a {\it stable spherical fibration} over a space $X$, we mean a map
\[
\xi\: X\to BG\, .
\]
One says that $\xi$ is {\it trivializable} if the map extends to the cone $C(X)$. A choice of extension is called a
{\it trivialization}.

An orientation sheaf $L$ on $X$ associated with $\xi$ is defined by the first Stiefel-Whitney 
class $w = w_1(\xi) \in H^1(X;\Bbb Z_2)$ which
is defined  as the pullback of the generator of  $H^1(BG;\Bbb Z_2) = \Bbb Z_2$.
Note that  $w$ may be regarded as a homomorphism $\pi\to\{\pm 1\}$
called the {\it orientation character}, where $\pi$ is the fundamental group of $X$.

To relate the above notion of stable spherical fibration to the classical one, 
let $p\: EG\to BG$ be a universal classifying fibration in which $EG$ is a free contractible
$G$-space. Form the fiber product
 $\bar X = X \times_{BG} EG$; the latter is equipped with a free action of $G$. Let $S$ be the sphere spectrum. Then 
 $G$ acts on $S$ and we obtain a parametrized spectrum
  \[
 \bar X\times_G S \to X
 \]
 whose fibers are spherical. The {\it Thom spectrum} of $\xi$ is
 \[
X^\xi := (\bar X_+) \smsh_{G} S\, .
 \]
Similarly, if $\tilde X \to X$ is the universal cover then we have a stable spherical fibration $\tilde X\to X \to BG$
and an associated  Thom spectrum with $\pi$-action $\tilde X^\xi$.
 
Then the {\it Thom class} of $\xi$
 \[
 U \in H^0(X^\xi ; L) := H^0_{\pi}(\tilde X^\xi; L)
 \]
induces the Thom isomorphism
 \[
 \cup U\: H^\ast(X;L) @>  \cong >> H^\ast(X^\xi;\Bbb Z ) \quad \text{ and } \quad
 \cap U\: H_\ast(X^\xi;\Bbb Z) @>  \cong >> H_\ast(X;L) \, .
\]
Note that if $f\:Y \to X$ is a map, then the pullback $f^\ast\xi$ is the composition $\xi \circ f$.

\begin{notation} 
If $i\: A\to X$ is an inclusion, then we sometimes write $A^\xi$ for the 
Thom spectrum of the pullback $i^\ast \xi$.
\end{notation}

\subsection{The Spivak normal fibration} \label{subsubsec:Spivak} Let $P$ be a Poincar\'e duality space of dimension $d$ and
let 
\[
\nu := \nu_P \: P \to BG
\]  
be a  stable spherical fibration.

Then $\nu$ is  a {\it Spivak fibration} for $P$ if there exists a homotopy class $\alpha\in \pi_d(P^\nu)$
such that
\[
[P] := \alpha_\ast([S^d])\in H_{d}(P^\nu) \cong H_d(P; L)
\]
is a  fundamental class, where $[S^d] \in H_d(S^d;\Bbb Z)$ is the generator \cite{Spivak,Browder_pd-surgery, Klein_dualizing}. 
The homotopy class $\alpha$ is sometimes called the {\it normal invariant}.
The triple
\[
(P,\nu,\alpha)
\]
is called a {\it Spivak normal structure} \cite[def.~9.30]{Ranicki_ags}. 
The triple is unique up to stable fiber homotopy equivalence (whose homotopy class is itself unique)
\cite[I.4.19]{Browder}, \cite{Klein_dualizing}.

If $(P, \partial P)$ is a Poincar\'e pair, then one also has a Spivak normal structure
which can be defined by pulling back the Spivak fibration of the double $P \cup_{\partial P} P$
along the left inclusion $P \subset P\cup_{\partial P} P$. 
In this instance, the normal invariant $\alpha$ is a  homotopy class of pairs
$(D^{d},S^{d-1}) \to  (P^\nu,(\partial P)^\nu)$ or equivalently, a homotopy class
\[
S^d \to P^\nu/(\partial P)^\nu\, .
\]

\begin{lem} \label{lem:inherit} If $f\: P \to M$ is a codimension-zero relative Poincar\'e embedding, then there is a fiber homotopy equivalence $\nu_P \simeq f^\ast \nu_M$.
\end{lem}

\begin{proof} Let $(M,\nu_M,\alpha_M)$ be a Spivak normal structure.
Set $\nu := \nu_M$. Collapsing the complement data of an embedding to a point defines a degree one stable {\it umkehr map}
\[
f^!\: M^\nu/(\partial M)^\nu \to P^{f^\ast\nu}/(\partial P)^{f^\ast \nu}\, .
\]
Then the homotopy class of $f^!\circ \alpha_M\: S^d\to P^{f^\ast\nu}/(\partial P)^{f^\ast \nu}$ satisfies the criterion of a normal invariant.
The result then follows from the uniqueness of a Spivak normal structure.
\end{proof}

\section{The Hopf invariant}  \label{sec:Hopf-invariant}

Let $Y\in T_\ast(\pi)$ be a cofibrant object. We set
\[
Q(Y)  :=  \Omega^\infty \Sigma^\infty(Y) \in T_\ast(\pi) \, .
\]
That is, $Q(Y)$ is the direct limit of the based mapping spaces in $T_\ast$:
\[
\colim_{n\to \infty} F_{T_\ast}(S^n, S^n \smsh Y) = F^{\st}_{T_\ast}(S^0,Y)\, .
\]
Let
\[
D_2(Y) = (Y \smsh Y) \smsh_{\Bbb Z_2} (E\Bbb Z_2)_+ \in T_\ast(\pi)
\]
be the {\it quadratic construction} of $Y$, where $E\Bbb Z_2$ is 
$S^\infty$ with the free $\Bbb Z_2$-action, $(E\Bbb Z_2)_+$ is 
$E\Bbb Z_2$ with a disjoint basepoint, and $(Y \smsh Y) \smsh_{\Bbb Z_2} (E\Bbb Z_2)_+$
is the orbit space of $\Bbb Z_2$ acting diagonally on $(Y \smsh Y) \smsh (E\Bbb Z_2)_+$.

The (equivariant stable) Hopf invariant \cite{Klein_Cartan} (see also \cite[defn.~7.8, 7.9]{Crabb-Ranicki}) is a natural
transformation
\[
H\: \{X,Y\}_{T_\ast(\pi)} \to \{X,D_2(Y)\}_{T_\ast(\pi)}
\]
that is induced by a
 $\pi$-equivariant  map
\[
h\: Q(Y) \to Q(D_2(Y)) \, 
\]
which is natural in $Y$.

The composition
\[
Y @> j >> Q(Y) @> h >> Q(D_2(Y))
\]
has a preferred natural null-homotopy, where $j$ is the inclusion.

\subsection{Properties}
For equivariant stable maps $f,g\: A\to B$, the {\it cup product}
\[
f\cup g \in \{A, B\smsh B\}_{T_\ast(\pi)}
\]
is the stable composition
\[
A @> \Delta_A >> A \smsh A @> f\smsh g >> B \smsh B\, ,
\]
where $\Delta_A$ is the reduced diagonal.
Let
\[
f\cup_2 g \in \{A,D_2(B)\}
\] 
be defined as  the stable composition
\[
A  @> f\cup g >> B \smsh B \to D_2(B)\, ,
\]
where the second displayed map is the equivariant homotopy class defined by the diagram
\[
B \smsh B @<\sim << (B \smsh B) \smsh (E\Bbb Z_2)_+ @>>> D_2(B)\, ,
\]
in which the left map is induced by the constant map $E\Bbb Z_2\to \ast$ and the right map is the quotient map defined by 
taking coinvariants. 

The double
cover of pairs 
\[
(B^{\times 2} \times E\Bbb Z_2, B^{\vee 2}\times E\Bbb Z_2)\to (B^{\times 2}  \times_{\Bbb Z_2} E\Bbb Z_2, B^{\vee 2} \times_{\Bbb Z_2} E\Bbb Z_2)
\]

induces a stable transfer map
\[
((B^{\times 2}  \times_{\Bbb Z_2} E\Bbb Z_2)_+, (B^{\vee 2}  \times_{\Bbb Z_2} E\Bbb Z_2)_+) \to ((B^{\times 2} \times E\Bbb Z_2)_+, (B^{\vee 2}\times E\Bbb Z_2)_+)
\]
\cite[ch.~4]{Adams_infinite},
which upon passing to quotients yields the  (equivariant stable) {\it transfer map}
\[
\tr\: D_2(B) \to B\smsh B\, .
\] 
\begin{rem} The above transfer map can also be defined on spectra as the norm map followed by the forgetful map:
\[
\Sigma^\infty D_2(B) @>>> (\Sigma^\infty (B\smsh B))^{h\Bbb Z_2} @>>> \Sigma^\infty (B\smsh B)\, .
\]
\end{rem}

For the next two results, we refer the reader to \cite[Add.~D]{Klein_Cartan}.

\begin{prop} \label{prop:Hopf-properties} The Hopf invariant satisfies:
\begin{enumerate}[(1).]
\item (Cartan Formula). $H(f+g) = H(f) + H(g) + f \cup_2 g$; 
\item (Transfer Formula). $\tr H(f)= (f\smsh f) \circ \Delta_A - \Delta_B \circ f$;
\item (Composition Formula). $H(g\circ f) = H(g)\circ f + D_2(g) \circ H(f)$.
\end{enumerate}
\end{prop}

Let $A, B\in T_\ast(\pi)$ be cofibrant and let 
\[
E\: [A,B]_{T_\ast(\pi)} \to \{A,B\}_{T_\ast(\pi)}
\]
denote the stabilization map. If $f\in [A,B]_{T_\ast(\pi)}$, then $H(E(f)) = 0$. Conversely, we have the following:

\begin{lem}\label{lem:EHP-property}
Assume that $B$ is $r$-connected, and $\dim_\ast A \le 3r+1$. Let $\alpha\in \{A,B\}_{T_\ast(\pi)}$.
If $H(\alpha) = 0$, then $\alpha = Ef$ for some $f\in [A,B]_{T_\ast(\pi)}$. \end{lem}

\begin{rem} \label{rem:EHP-property} 
Lemma \ref{lem:EHP-property} follows immediately from the statement that
there is a $(3r+1)$-Cartesian square
\[
\xymatrix{
B \ar[r] \ar[d] & C(B) \ar[d] \\
Q(B) \ar[r] & Q(D_2(B))
}
\]
in which $C(B)$ is  the cone on $B$ (cf.~\cite[cor.~C]{Klein_Cartan}).
\end{rem}

\subsection{The coaction} Suppose
that
\begin{equation} \label{eqn:cofiber-sequence}
A\to X \to X/A
\end{equation}
is a cofibration sequence of $T_\ast(\pi)$. Let
\[
c\in [X/A,X/A \vee \Sigma A]_{T_\ast(\pi)}
\]
denote the coaction element represented by the diagram
\[
X/A @< \simeq << X \cup CA @>>> X/A \vee \Sigma A\, ,
\]
where the map $X \cup CA @>>> X/A \vee \Sigma A$ is the effect of pinching $A$ to a point.
Let
\[
\delta\in \{X/A,\Sigma A\}_{T_\ast(\pi)}
\] 
be the equivariant stable homotopy class of the Barratt-Puppe extension of
the cofiber sequence \eqref{eqn:cofiber-sequence} as defined by the diagram
\[
X/A @< \simeq << X \cup CA @>>> \Sigma A\, ,
\]
in which the map $X \cup CA @>>> \Sigma A$ is obtained by pinching $X$ to a point.

Then we have an action
\[
\{X/A,Y\}_{T_\ast(\pi)} \times \{\Sigma A,Y\}_{T_\ast(\pi)} \to \{X/A,Y\}_{T_\ast(\pi)}
\]
given by $(f,\alpha) \mapsto f \sharp \alpha$, where $f\sharp \alpha$ is the composition
\[
X/A @> c >> X/A \vee \Sigma A @> f \vee \alpha >> Y \vee Y @> \text{fold} >> Y\, .
\] 

\begin{lem}[Action Lemma] \label{lem:coaction} With respect to the above,
$H(f\sharp \alpha) \in \{X/A,D_2(Y)\}_{T_\ast(\pi)}$
is given by
\[
H(f) + H(\alpha)\circ \delta \, .
\]
\end{lem}

\begin{proof} The homotopy class $f\sharp \alpha$ is given by
\[
(f \circ p_1 + \alpha \circ p_2) \circ c\, ,
\]
where $p_1 \: X/A \vee \Sigma A @>p_1 >> X/A$ and $p_2 \: X/A \vee \Sigma A @>p_2 >> \Sigma A$
are the projections. Then by the composition formula, the Cartan formula, and the fact that $p_1,p_2$ and $c$ are 
defined unstably, we have
\begin{align*}
H((f \circ p_1 + \alpha \circ p_2) \circ c) &= H(f \circ p_1 + \alpha \circ p_2)\circ c\, , \\
&= H(f)\circ p_1\circ c + ((f \circ p_1) \cup_2 (\alpha \circ p_2)) \circ c + H(\alpha)\circ p_2\circ c\, ,  \\
&= H(f) +   ((f \circ p_1) \cup_2 (\alpha \circ p_2)) \circ c+  H(\alpha) \circ \delta\, ,
\end{align*}
where in the last line we have used the fact that $\delta = p_2\circ c$.

Consequently, it is enough to show that $((f \circ p_1) \cup_2 (\alpha \circ p_2)) \circ c$ is trivial. By definition, the latter is given by the composition
\begin{equation*}
\mbox{\Small $\xymatrix{
X/A \ar[r]^(.4){c}  & X/A \vee \Sigma A \ar[r]^(.33){\Delta} &  (X/A \vee \Sigma A) \smsh (X/A \vee \Sigma A)\ar[d]^{p_1 \smsh p_2}  \\
&& X/A \smsh \Sigma A \ar[r]^{f\smsh \alpha} &  Y \smsh Y \ar[r] & D_2(Y)\, ,
}$}
\end{equation*}
where $\Delta$ is the diagonal. Since $\Delta \circ c$ coincides with the composition
\[
X/A @> \Delta >> X/A \smsh X/A @> c\smsh c >> (X/A \vee \Sigma A) \smsh (X/A \vee \Sigma A)\, ,
\]
we infer
\begin{align*}
(p_1 \smsh p_2) \circ \Delta \circ c &= (p_1 \smsh p_2) \circ (c\smsh c) \circ \Delta \, , \\
& = (p_1 \circ c) \smsh (p_2\circ c) \circ \Delta \, ,\\
&  = (1 \smsh \delta) \circ \Delta \, ,\\
&= 0\, ,
\end{align*}
where we have observed that $(1 \smsh \delta) \circ \Delta$ is trivial since it coincides with the composition
\[
X/A @> c >> X/A \vee \Sigma A @>\ast >> X/A \smsh \Sigma A\, . 
\]
The result follows.
\end{proof}

\section{The relative Hopf invariant} \label{sec:rel-Hopf}

Let $B \in T_\ast(\pi)$ be cofibrant, let $Y \in T_B(\pi)$ be an object, and
let $A\to X$ be a cofibration in $T_B(\pi)$.
Then by Definition  \ref{defn:function-space}, one has
\[
\{X,Y; A \rel B\}_{T_\ast(\pi)}  = \{X,Y; A\}_{T_B(\pi)} \, .
\]
Equivalently,
\[
\{X,Y; A \rel B\}_{T_\ast(\pi)} = [(X,A),(Q(Y),Y)]_{T_B(\pi)}\, .
\]
 \begin{ex} When $A = B$, we have
  \[
\{X,Y;B \rel B\}_{T_\ast(\pi)} = \{X,Y \rel B\}_{T_\ast(\pi)} =  \{X,Y\}_{T_B(\pi)} \, .
 \]
 When $B = \ast$, we have
 \[
 \{X,Y;A \rel \ast\}_{T_\ast(\pi)} =  \{X,Y;A\}_{T_\ast(\pi)} \, .
 \]
 \end{ex}

Since $Y$ inherits a basepoint from $B$, we infer that
the composition
\[
Y @>>> Q(Y) @>h>> Q(D_2(Y))
\]
is defined and has a preferred null homotopy. Consequently, it determines a commutative
 square of $T_\ast(\pi)$: 
\begin{equation} \label{eqn:EHP}
\xymatrix{
Y \ar[r] \ar[d] & CY \ar[d]\\
Q(Y) \ar[r]_(.4)h & Q(D_2(Y)),
}
\end{equation}
where $CY$ is the cone on $Y$. If $Y$ is $r$-connected, then the square is $(3r+1)$-Cartesian \cite[thm.~1.11]{Milgram}.\footnote{In \cite{Milgram} an exact sequence is constructed which arises from the map $Y \to \Omega^k \Sigma^k Y$, so one must take $k \to \infty$ to achieve the connectivity of the displayed 
diagram. The  result also follows from a version of the Snaith splitting \cite{Cohen_snaith}, or by a functor calculus argument similar to  \cite[ex.~1.20]{Goodwillie_calc3}. See also \cite[Add.~D]{Klein_Cartan} in the equivariant case.}

Using the model structure, choose a factorization
\[
\begin{tikzcd} 
CY \ar[r,"\subset"] & Q(D_2(Y))' \ar[r,"\sim"] & Q(D_2(Y))
\end{tikzcd}
\]
where the first map is a cofibration and the second is both a fibration and a weak equivalence.
Then the induced function of homotopy classes
\[
[(X,A),(Q(Y),Y)]_{T_B(\pi)} \to [(X,A),(Q(D_2(Y))',CY)]_{T_B(\pi)} 
\]
arising from the  square is therefore a function of stable homotopy classes 
\begin{equation} \label{eqn:rel-Hopf}
H_A \: \{X,Y; A \rel B\}_{T_\ast(\pi)} \to  \{X/A,D_2(Y)\}_{T_\ast(\pi)}\, .
\end{equation}
We call $H_A$  the {\it relative Hopf invariant}.


Observe that the composition
\[
[X,Y \rel B]_{T_\ast(\pi)} @> E >> \{X,Y;A \rel B\}_{T_\ast(\pi)} @>H_A >>  \{X/A,D_2(Y)\}_{T_\ast(\pi)}
\]
is trivial, where  $E$ is the stabilization map of Definition \ref{defn:function-space}. Recall the notion of connectivity and relative dimension from Definition \ref{defn:conn-dim}.

\begin{lem} \label{lem:desuspension} Assume $Y \in T_{\ast}(\pi)$ is $r$-connected and $\dim_A X \le 3r+1$.
Let $f\in \{X,Y; A\rel B\}_{T_\ast(\pi)} $. Then $f$ is in the image of $E$ if $H_A(f) = 0$.
\end{lem}

\begin{proof}  
This is a straightforward consequence of the fact that the square
\eqref{eqn:EHP} is $(3r+1)$-Cartesian.
\end{proof}

\subsection{Desuspending maps of retractive spaces} \label{subsec:retract-desuspend}
The obstruction to desuspending a fiberwise map in the metastable range
necessitates the introduction of a fiberwise Hopf invariant  which we prefer to avoid for technical reasons. 
Fortunately, when it comes to Poincar\'e surgery, the degree of instability that  arises is just one dimension outside of
the stable range, i.e., the ``fringe'' dimension. In the fringe dimension,  one can replace fiberwise homotopy theory with
equivariant homotopy theory with respect to the fundamental group. This is implemented below.

Let $X \in T$ be cofibrant and connected. 
Let $\tilde X\to X$ denote the universal cover with group of deck transformations $\pi$. If $Y \in T_{/X}$ is an object, then
the fiber product
\[
 \tilde Y = Y \times_X \tilde X \in T(\pi)_{/\tilde X}
\]
is a space over $\tilde X$ with $\pi$-action.
Similarly, if $Y\in R(X)$ is an object, then $\tilde Y \in R(\tilde X,\pi)$.
Then the quotient
\[
Y^{\sharp} := \tilde Y/\tilde X \in T_\ast(\pi)
\]
is a  based $\pi$-space which is cofibrant if $Y$ is cofibrant.

\begin{ex} If $U \in T_{/X}$ is an object, then $U^+ = U \amalg X \in R(X)$, and
\[
(U^+)^\sharp  = \tilde U \amalg \ast = \tilde U_+\, .
\]
\end{ex}

The operation $Y \mapsto Y^\sharp$ defines a functor 
\[
({-})^\sharp\: R(X) \to T_\ast(\pi)\, .
\]
In particular, for objects $Y,Z\in R(X)$, and $A\to Y$ a cofibration, we obtain  functions of hom-sets
\[
({-})^\sharp \: [Y,Z]_{R(X)} \to [Y^\sharp,Z^\sharp]_{T_\ast(\pi)} \quad \text{ and } \quad ({-})^\sharp \: \{Y,Z;A\}_{R(X)} \to \{Y^\sharp,Z^\sharp; A^\sharp\}_{T_\ast(\pi)} \, .
\]

Consider the stabilization map
\[
 E\: [Y,Z]_{R(X)} \to \{Y,Z; A\}_{R(X)}
 \]
defined by iterated application of the reduced fiberwise suspension functor $\Sigma_X$ of \S\ref{subsubsec:suspension}.

\begin{lem}[Fringe Lemma] \label{lem:fringe} Suppose that $Z$ is $r$-connected and cofibrant. Then 
an element  $f\in \{Y,Z; A\}_{R(X)}$ is in the image of $E$ if $\dim_A Y \le 2r+1$. 

If $\dim_A Y = 2r+2$, and $r \ge 2$, then $f$ is in the image of $E$
if and only if 
\[
H_{A^{\sharp}}(f^\sharp) \in \{Y^\sharp/A^\sharp,D_2(Z^\sharp)\}_{T_\ast(\pi)}
\]
is trivial.
\end{lem}

\begin{proof} The first part is a direct consequence of the fiberwise Freudenthal suspension theorem
\cite[prop.~9.3]{James_fibrewise}, \cite[note 4.10]{Klein_haef}.

For the second part, we recall the well-known fact: Let $U \to X$ be a universal principal bundle with topological structure group $G$; then $U$ is contractible and
$X \simeq BG$. 

Without loss of generality, we may assume that $G$ is cofibrant. Let $T_{\ast}(G)$ be the model category of based left $G$-spaces.
There is a functor
\[
b\: R(X) \to  T_{\ast}(G)\, ,
\]
where  
\[
b(Y) = \bar Y :=  (Y \times_X U)/(X\times_X U)\ .
\]
The functor $b$ induces an equivalence of homotopy categories
\[
\text{ho}R(X) \cong \text{ho}T_{\ast}(G)
\]
(cf.~\cite[proof of prop.~2.1.4]{Wald_1126}). 

The functor $b$ preserves cofibrations as well as the dimension of objects.
We also claim that it preserves connectivity. To prove the claim,
observe that $b$ is a composition of the form
\[
R(X) @> p >> R(U,G) @> q >> T_\ast(G), 
\]
where $R(U,G)$ is the category of $G$-equivariant retractive spaces over $U$,
and $p(Y) = Y\times_X U$ and $q(Z) = Z/U$.
The functor $p$ is an equivalence of categories by \cite[lem.~2.1.3]{Wald_1126}, 
and it clearly preserves connectivity since it is given by a base change. 
In the proof of \cite[prop.~2.1.4]{Wald_1126} it is shown that $q$ induces an equivalence
of homotopy categories, where an inverse $r\: T_\ast(G) \to R(U,G)$ up to natural weak equivalence is defined by
$r(W) = W\times U$,  the latter equipped with the diagonal $G$-action.
Since $U$ is contractible,  $r$ preserves connectivity. Therefore $q$ preserves connectivity and we conclude
that the functor $b$ also preserves connectivity, establishing the claim.

Let $G\to \pi_0(G) = \pi$ be the homomorphism to path components. Then the functor $\sharp$ factors as
\[
R(X)@>>> T_{\ast}(G) @>I_G^\pi >> T_{\ast}(\pi) \, ,
\]
where the first functor is given by $Y \mapsto \bar Y$ and the second functor is given by induction $B \mapsto B\smsh_G (\pi_+)$.

By obstruction theory, it suffices to show that for an $r$-connected cofibrant object $C \in T_{\ast}(G)$, the commutative
square
\[
\xymatrix{
C \ar[r] \ar[d] & C\smsh_G(\pi_+) \ar[d]\\
Q(C) \ar[r] & Q(C\smsh_G(\pi_+))
} 
\]
is $(2r+2)$-Cartesian (since the left vertical map induces stabilization in $\text{ho}T_\ast(G)$ and the right vertical map induces
stabilization in $\text{ho}T_\ast(\pi)$). To see this, first observe that the map $C\to C\smsh_G(\pi_+)$ is $(r+2)$-connected, since it is given by smashing
the 1-connected map $G_+\to \pi_+$ with $C$ and taking $G$-orbits. Consequently, the map
\begin{equation} \label{eqn:quad-construction-map}
Q(D_2(C)) \to Q(D_2(C\smsh_G(\pi_+)))
\end{equation}
is $(2r+3)$-connected since it is given by applying homotopy $\Bbb Z_2$-orbits and then $Q$ to the $(2r+3)$-connected map
$C \smsh C \to (C\smsh_G(\pi_+)) \smsh (C\smsh_G(\pi_+))$.
For the rest of the proof, we set $C' = C\smsh_G(\pi_+)$. 

Consider the cubical diagram

\begin{center}
\begin{tikzpicture}[>=stealth, scale=.7]

\node (C)   at (0,4) {$C$};
\node (CC)  at (6,4) {$C(C)$};

\node (Cp)  at (2.7,2.5) {$C'$};
\node (CCp) at (8.7,2.5) {$C(C')$};

\node (QC)  at (0,0) {$Q(C)$};
\node (QD)  at (6,0) {$Q(D_2(C))$};

\node (QCp) at (2.7,-1.5) {$Q(C')$};
\node (QDp) at (8.7,-1.5) {$Q(D_2(C'))$};

\draw[->] (C) -- (CC);
\draw[->] (C) -- (QC);
\draw (QC) -- (2.6,0); 
\draw[->] (2.8,0) -- (QD);

\draw (CC) -- (6,2.6);
\draw[->] (6,2.4) -- (QD);

\draw[->] (Cp) -- (CCp);

\draw[->](Cp) --  (QCp);

\draw[->] (QCp) -- (QDp);
\draw[->] (CCp) -- (QDp);

\draw[->] (C) -- (Cp);
\draw[->] (CC) -- (CCp);
\draw[->] (QC) -- (QCp);
\draw[->] (QD) -- (QDp);

\end{tikzpicture}
\end{center}
where $C(C)$ and $C(C')$ are the cones on $C$ and $C'$ respectively. The back and front faces of the cube are $(3r+1)$-Cartesian (cf.~Remark \ref{rem:EHP-property}).
Since the map \eqref{eqn:quad-construction-map} is $(2r+3)$-connected, the right face of the cube is $(2r+2)$-Cartesian. 
Hence, the left face of the cube is $\min(3r,2r+2)$-Cartesian. Since we are assuming $r \ge 2$, we have $3r \ge 2r+2$, so that the left 
face of the cube is $(2r+2)$-Cartesian, as was to be proved.
\end{proof}

An inspection of the proof of Lemma \ref{lem:fringe} shows that it establishes a bit more:

\begin{cor}\label{cor:fringe} 
Assume $\dim_A Y = 2r+2$, $Z$ is $r$-connected and cofibrant, and  $r \ge 2$. Then the commutative diagram of derived function spaces
\[
\xymatrix{
\cal F_{R(X)}(Y,Z) \ar[r] \ar[d] & \cal F_{T_{\ast}(\pi)}(Y^\sharp,Z^\sharp) \ar[d] \\
\cal F^{\st}_{R(X)}(Y,Z ; A) \ar[r] & \cal F^{\st}_{T_{\ast}(\pi)}(Y^\sharp,Z^\sharp; A^\sharp )
}
\]
is 0-Cartesian.
\end{cor}

\begin{rem} In applications, we  typically assume that $A$ is cofibrant. It follows that $Y$ is cofibrant and therefore so is $Y^\sharp$.
Consequently, we may use $F_{T_{\ast}(\pi)}(Y^\sharp,Z^\sharp)$ in place
of  $\cal F_{T_{\ast}(\pi)}(Y^\sharp,Z^\sharp)$.
\end{rem}

\subsection{Gluing} \label{subsec:gluing} Fix a cofibrant object $A\in T_\ast(\pi)$  and cofibrant objects $X,Y \in T_A(\pi)$.
For a subset $U \subset \Bbb R$, we set
\[
A_U := (A \times U)/(\ast \times U) = A \smsh (U_+) \in T_\ast(\pi)\, .
\]
By convention, the basepoint of $I = [0,1]$ is $0 \in I$.
Then $A_{\partial I} = A_0 \vee A_1$.

Consider the set
\[
\{A_I,Y; A_{\partial I} \rel A_0\}_{T_\ast(\pi)}\, 
\] 
consisting of 
the homotopy classes of equivariant stable maps $A_I \to Y$ which are fixed on $A_0$ and which are unstable on
$A_1$.

\begin{defn} \label{defn:gluing}
The {\it gluing operation} $\Theta$:
\[
 \scalebox{.9}{$\{X,Y \rel  A\}_{T_\ast(\pi)} \times \{A_I,Y; A_{\partial I} \rel A_0\}_{T_\ast(\pi)}\to 
\{ X\cup_{A_0} A_I,Y; A_1 \}_{T_\ast(\pi)} = 
\{ X,Y; A \}_{T_\ast(\pi)}$}
\]
is defined by \[
\Theta(f,\alpha)= f\cup_{A_0} \alpha\, ,
\] 
i.e., we amalgamate $f$ and $\alpha$ along $A_0$ and observe that the resulting map $f\cup_{A_0} \alpha$ is then unstable when restricted to $A_1$. 
\end{defn}

\begin{rem} If the structure map $A\to Y$ is constant, then the gluing operation simplifies to 
\[
\{X/A,Y\}_{T_\ast(\pi)} \times \{CA_1, Y; A_1\}_{T_\ast(\pi)}\to \{X,Y; A\}_{T_\ast(\pi)}\, ,
\]
where $CA= A_I/A_0$ is the cone on $A$. In this case, there is
a preferred function
\begin{equation} \label{eqn:preferred-homo}
 \{CA_1, Y; A_1\}_{T_\ast(\pi)} \to \{\Sigma A,D_2(Y)\}_{T_\ast(\pi)} \, .
\end{equation}
Assume
 $Y$ is $r$-connected. Then the square 
  \eqref{eqn:EHP} is  $(3r+1)$-Cartesian. Consequently, if $\dim A \le 3r+1$, then \eqref{eqn:preferred-homo}
 is surjective. It is injective if $\dim A \le 3r$. \end{rem}

Recall that $\delta\: X/A \to \Sigma A$ is the Barratt-Puppe extension of the cofibration 
sequence $A\to X \to X/A$.

\begin{lem}[Gluing Lemma] \label{lem:gluing} 
$H_{A_1}(\Theta(f,\alpha)) = H_A(f) + H_{A_{\partial I}}(\alpha) \circ \delta$.
\end{lem}

\begin{proof} By definition, $H_{A_1}(\Theta(f,\alpha))  := H_{A_1}(f\cup_{A_0} \alpha)$ factorizes as
\[
\scalebox{.9}{$(X \cup_{A_0} A_I)/A_1 @> p >> X/A_0 \vee A_I/A_{\partial I} @>H_{A}(f) \vee H_{A_{\partial I}}(\alpha) >> D_2(Y) \vee D_2(Y) @> \text{fold} >> D_2(Y)\, ,$}
\]
where the map $p$ pinches $A_0$ to a point. Note that $(X \cup_{A_0} A_I)/A_1  = X \cup_{A_0} CA_0$ and $A_I/A_{\partial I} = \Sigma A$. 
Moreover, $p$ is identified with the coaction map $c\: X/A \to X/A \vee \Sigma A$
of the cofiber sequence $A\to X \to X/A$. The result follows since $c$ is stably homotopic to the sum $i_1+ i_2\circ \delta$,
where $i_1\: X/A \to X/A \vee \Sigma A$ and $i_2\: \Sigma A \to X/A \vee \Sigma A$ are the inclusions to the wedge summands.
\end{proof}

\begin{rem} The Action Lemma can be viewed as a particular case of the Gluing Lemma by 
considering  a constant map $A_1\to Y$ and the forgetful function
\[
 \{\Sigma A,Y\} = \{CA_1, Y\}_{T_{A_1}(\pi)}  \to \{CA_1, Y; A_1\}_{T_\ast(\pi)}\, .
 \]
We leave the details to the reader.
\end{rem}

\section{Proof of Theorem \ref{bigthm:compression}} \label{sec:compression}

We recall the set-up: Let
\[
(P;\partial_0 P,\partial_1 P) \quad \text{ and } \quad (M;\partial_0 M,\partial_1 M)
\]
be  $d$-dimensional Poincar\'e triads with
$(\partial_0 P,\partial_{01} P) = (\partial_0 M,\partial_{01} M)$.
Let
\[
P \times J \to M \times I
\]
be an embedding relative to $(\partial_0 P) \times J$ and with associated  codimension-one splitting
\[
(M\times I,P\times J, W) \, . 
\]
Let $\lambda\in [P/\!\!/\partial_0 P,W]_{R(M)}$ be the link of the embedding. Assume that $\lambda$ is trivial.

\begin{proof}[Proof of Theorem \ref{bigthm:compression}]  The argument
is essentially the same as that of \cite[thm.~A]{Klein_compression} with some attention given
to the partial boundary. To
 avoid clutter, we introduce the following notation: If $X$ is a set and
$S\subset \Bbb R$ is a subset, then we let $X_S = X\times S$. Similarly, if $s\in \Bbb R$,
we let $X_s = X \times \{s\}$.\footnote{The author is aware that the current notation conflicts with the way it is used in 
\S\ref{subsec:gluing}.  This should not cause confusion as the contexts are distinct.}

The given embedding $\partial_0 P \to \partial M$ 
and a choice of null homotopy of a representative of $\lambda$ determine a map
\[
\phi\: (M_0 \cup P_{[0,1/3]} , M_0 \cup \partial_0 P_{[0,1/3]} \cup P_{1/3} ) \to (W,\partial W)\,
\]
which restricts to an embedding $M_0 \cup \partial_0 P_{[0,1/3]} \cup P_{1/3}  \to \partial W$. 
In the display, $M_0 \cup P_{[0,1/3]}$ is identified with the mapping cylinder of the 
composite 
\[
P_{1/3} \to M_I @> \text{project} >>  M_0\, .
\]
The homotopy pushout diagram (cf.~Remark \ref{rem:alternative-display}) associated with the embedding $M_0 \cup \partial_0 P_{[0,1/3]} \cup P_{1/3}  \to \partial W$ is:
\[
\xymatrix{
(\partial_1 M)_0  \cup (\partial_{01} P)_{[0,1/3]} \cup (\partial_1 P)_{1/3}\ar[d]  \ar[r] & (\partial_1 M)_I \cup M_1 \cup   (\partial_1 P)_J \cup P_{2/3} \ar[d]  \\
M_0 \cup \partial_0 P_{[0,1/3]} \cup P_{1/3} \ar[r] & \partial W
}
\]
(compare  \cite[p.~330, first diagram]{Klein_compression}).

The map $M_0 \cup P_{[0,1/3]} \to W$ is $(d-p-1)$-connected, since the composition
$M_0 \cup P_{[0,1/3]} \to W \to M_I$ is a weak equivalence and the map $W\to M_I$ is $(d-p)$-connected (the latter
follows from the fact that $S_P \partial_1 P \to P_I$ is $(d-p)$-connected). It is also readily verified that 
$M_0 \cup P_{[0,1/3]}$ is obtained from $M_0 \cup \partial_0 P_{[0,1/3]} \cup P_{1/3}$ by attaching cells of dimension $\le p+1$.
Consequently, by \cite[thm.~A]{Klein_haef2}, the map $\phi$ underlies a relative embedding
\[
(\bar M, M_0 \cup \partial_0 P_{[0,1/3]} \cup P_{1/3}) \to (W,\partial W)
\]
for a suitable codimension-one splitting 
\begin{equation} \label{eqn:splitting-for-thm}
(\bar M, M_0 \cup \partial_0 P_{[0,1/3]} \cup P_{1/3},A)\, ,
\end{equation} where the map
$M_0 \to \bar M$ is a weak equivalence and $A \to \bar M$ is $2$-connected (cf.~Figure \ref{fig:to-be-compressed-in-proof}).

\begin{figure} 
\centering
\begin{tikzpicture}[scale=1]
\node[
  cylinder,
  draw=blue!70,
  minimum width=3cm,
  minimum height=3cm,
  shape border rotate=90,
  cylinder uses custom fill,
  cylinder body fill=blue!20,
  cylinder end fill=blue!40,
  fill opacity=0.25
] (big) at (0,0) {};
\draw[dotted,
  blue!60,
  fill=blue!30,
  fill opacity=0.35
]
  ([yshift=0.17cm]big.south) ellipse[x radius=1.5cm, y radius=0.17cm];
\node[
  cylinder,
  draw=red,
  minimum width=0.8cm,
  minimum height=1.5cm,
  shape border rotate=90,
  cylinder uses custom fill,
  cylinder body fill=red!20,
  cylinder end fill=red!30,
  fill opacity=0.2
] (small) at (1.1,-0.2) {};
\draw[red, dotted, line width=.8pt]
  ([yshift=0.60cm]small.center) -- ([yshift=-0.65cm]small.center);
\node[font=\relsize{-3}\smaller] at (1.05,-.1) {$P_J$};
\node[scale=0.5] at (1.05,.5) {$P_{2/3}$};
\node[font=\relsize{-3}\smaller] at (0,-.1) {$C$};
\node[font=\relsize{-3}\smaller] at (0,1.45) {$M_1$};
\node[font=\relsize{-3}\smaller] at (1.05,-1) {$A$};
      \draw[blue!60] (1.1,.3) arc[start angle=-90, end angle=0, x radius=.4cm, y radius=.25cm];
    \draw[blue!60] (1.1,.2) arc[start angle=-90, end angle=0, x radius=.4cm, y radius=.25cm];
  \draw[blue!60] (1.1,.1) arc[start angle=-90, end angle=0, x radius=.4cm, y radius=.25cm];
  \draw[blue!60] (1.1,0) arc[start angle=-90, end angle=0, x radius=.4cm, y radius=.25cm];
        \draw[blue!60] (1.2,-.07) arc[start angle=-90, end angle=0, x radius=.3cm, y radius=.2cm];
    \draw[blue!60] (1.2,-.17) arc[start angle=-90, end angle=0, x radius=.3cm, y radius=.2cm];
    \draw[blue!60] (1.1,-.3) arc[start angle=-90, end angle=0, x radius=.4cm, y radius=.25cm];
    \draw[blue!60] (1.1,-.4) arc[start angle=-90, end angle=0, x radius=.4cm, y radius=.25cm];
        \draw[blue!60] (1.1,-.5) arc[start angle=-90, end angle=0, x radius=.4cm, y radius=.25cm];
            \draw[blue!60] (1.1,-.6) arc[start angle=-90, end angle=0, x radius=.4cm, y radius=.25cm];
                  \draw[blue!60] (1.1,-.7) arc[start angle=-90, end angle=0, x radius=.4cm, y radius=.25cm];
                        \draw[blue!60] (1.1,-.8) arc[start angle=-90, end angle=0, x radius=.4cm, y radius=.25cm];
\draw[gray, thick] (0.7,-0.74) .. controls (-0.4,-1.3) .. (-1.5,-1.3);
  \draw[gray] (-.7,-1.39) arc[start angle=-90, end angle=-60, x radius=4.4cm, y radius=1.1cm];
    \draw[gray] (-1,-1.3) arc[start angle=-90, end angle=-60, x radius=5.0cm, y radius=1cm];
        \draw[gray] (-.35,-1.2) arc[start angle=-90, end angle=-60, x radius=3.7cm, y radius=1cm];
              \draw[gray] (-.06,-1.1) arc[start angle=-90, end angle=-60, x radius=3.1cm, y radius=1cm];
                \draw[gray] (0.2,-1.0) arc[start angle=-90, end angle=-60, x radius=2.6cm, y radius=1cm];
                   \draw[gray] (0.4,-.9) arc[start angle=-90, end angle=-60, x radius=1.4cm, y radius=.6cm];
                       \draw[gray] (0.6,-.8) arc[start angle=-90, end angle=-60, x radius=.4cm, y radius=.22cm];
\end{tikzpicture}
\caption{The embedding $P_J\to M_I$ displayed with $(C,A)$ (cf.~diagram \eqref{eqn:embedding-pushout-first}).}
\label{fig:to-be-compressed-in-proof}
\end{figure}

Consequently, the commutative homotopy pushout associated with this relative embedding is of the form
\begin{equation} \label{eqn:embedding-pushout-first}
  \resizebox{0.9\textwidth}{!}{%
  \xymatrix{
 (A, (\partial_1 M)_0  \cup (\partial_{01} P)_{[0,1/3]} \cup (\partial_1 P)_{1/3}) \ar[d]  \ar[r] & (C,(\partial_1 M)_I \cup M_1 \cup   (\partial_1 P)_J \cup P_{2/3}) \ar[d]  \\
(\bar M, M_0 \cup \partial_0 P_{[0,1/3]} \cup P_{1/3}) \ar[r] & (W,\partial W)
}}
\end{equation}
(compare  \cite[p.~330, second diagram]{Klein_compression}), and after identifying $\bar M$ with $M$ we obtain a commutative square 
\begin{equation} \label{eqn:embedding-pushout}
 \xymatrix{
(\partial_1 P,\partial_{01} P) \ar[r] \ar[d] & (A,\partial_1 M)\ar[d]  \\
(P,\partial_0 P) \ar[r] & (M,\partial M) 
}
\end{equation} whose restriction to $\partial M$ displays the given boundary splitting $(\partial M, \partial_0 M,\partial_1 M)$
(compare \cite[p.~331]{Klein_compression}).
To complete the proof, it suffices to show that the latter square defines a relative Poincar\'e embedding
whose decompression coincides with the given embedding $P_J \to M_I$ up to concordance.

Using the identification
\[
\partial_0 P_{[0,1/3]} \cup_{(\partial_0 P)_{1/3}} P_{1/3 }\cup_{(\partial_1 P)_{1/3}} A \simeq P\cup_{\partial_1 P} A\, ,
\]
it follows that the decomposition of  $\partial\bar M$ arising from the splitting \eqref{eqn:splitting-for-thm} is identified with
\[
M_0 \cup_{\partial M} (P\cup_{\partial_1 P} A)\, .
\]
Observe that the map $P\cup_{\partial_1 P} A \to \bar M$ induces an isomorphism on fundamental groups by the van Kampen theorem, 
since both $\partial_1 P \to P$ and $A \to \bar M$ are $2$-connected.
Applying Lefschetz duality with respect to any coefficient system on $\bar M$, we see
\[
H^\ast(\bar M,A \cup_{\partial_1 P}  P) \cong H_{d+1-\ast}(\bar M,M_0) = 0\, ,
\]
so the map $P\cup_{\partial_1 P} A  \to \bar M$ is a
weak equivalence by the relative Hurewicz theorem. Therefore \eqref{eqn:embedding-pushout} is indeed a homotopy
pushout and defines a relative Poincar\'e embedding.

To complete the proof, we need to show that the embedding \eqref{eqn:embedding-pushout} decompresses to  the one we started with
up to concordance. The homotopy pushout square of the original embedding $P_J \to M_I$ is
\[
\xymatrix{
(S_P \partial_1 P,S_{\partial_0 P}\partial_{01} P) \ar[r] \ar[d] & (W,S_{M} \partial_1 M)\ar[d]  \\
(P_J,(\partial_0 P)_J) \ar[r] & (M_I,\partial (M_I)) \, .
}
\]
Observe that the homotopy pushout square of the decompression is
\[
\xymatrix{
(S_P \partial_1 P,S_{\partial_0 P}\partial_{01} P) \ar[r] \ar[d] & (S_M A,S_{M} \partial_1 M)\ar[d]  \\
(P_J,(\partial_0 P)_J) \ar[r] & (M_I,\partial (M_I)) \, .
}
\]
Note that the evident map
$C \to M$
is a weak equivalence, since there is a homotopy pushout square
\[
\xymatrix{
P \cup_{\partial_1 P} A \ar[r] \ar[d] & C \ar[d] \\
\bar M \ar[r] & M_I
}
\]
in which the left vertical map is a weak equivalence and weak equivalences are preserved under cobase change.
Consequently, the homotopy pushout squares \eqref{eqn:embedding-pushout-first} and \eqref{eqn:embedding-pushout} induce weak equivalences
\[
  \resizebox{0.9\textwidth}{!}{%
   $S_M A = M\times \{0\} \cup {A{\times} I} \cup M \times \{1\}  \simeq M\times \{0\} \cup {A{\times} I} \cup C \times \{1\}  \simeq  W\, .$
   }
\]
which are compatible with projection to $M$. Furthermore, by a somewhat tedious but straightforward check which we omit, 
the weak equivalences restrict to the identity on $\partial W = S_P \partial_1 P \cup_{S_{\partial_0 P} \partial_{01} P} S_M \partial_1 M$.
But this is precisely what is needed in order to show that the two embeddings  are concordant (see Remark \ref{rem:crucial-observation}).
\end{proof}

\begin{rem} \label{rem:crucial-observation} In the last part of the proof we used the following observation: 
For a fixed map $A\to B$, let $T(A\to B)$ be the category of spaces over $B$ and under $A$.
An object of  $T(A\to B)$ is a space $X$ that is equipped with a factorization $A\to X\to B$.
A morphism $X\to Y$ is a map of spaces that is compatible with the factorizations.

Suppose we are given two  embeddings from $P$ to $M$ with fixed underlying map 
$f\: P \to M$. Let $(C;\partial_1 P,\partial_1 M)$ and $(C';\partial_1 P,\partial_1 M)$
be their complements. Then 
\[
C,C'\in T(\partial_1 P \cup \partial_1 M \to M)
\] 
are objects.
Furthermore, the Poincar\'e embeddings are concordant if and 
only if there is a finite zig-zag of weak equivalences from $C$ to $C'$.
\end{rem}

\section{Proof of Theorem \ref{bigthm:Smale-Hirsch}} \label{sec:immersions}

Let $(P;\partial_0 P,\partial_1 P)$ and $(M; \partial_0 M,\partial_1 M)$ be Poincar\'e triads of dimension $d$ where $\partial_0 P = \partial_0 M$.
Recall from Definition \ref{defn:immersion-space} that
\[
I(P,M)
\]
is the space of  immersions from $P$ to $M$
 which are fixed on $\partial_0 P$. Similarly, $F(P,M)$ is the function space of maps $P \to M$ which are fixed on $\partial_0 P$.
There is a map 
\begin{equation} \label{eqn:immersion-map}
 I(P,M) \to F(P,M)\, ,
 \end{equation}
which assigns to an immersion its underlying map.

Let $f\in F(P,M)$ be a choice of basepoint, and let 
 \[
 I_{f}(P,M) 
 \]
 denote the homotopy fiber of \eqref{eqn:immersion-map} taken at the basepoint.
A point of $I_{f}(P,M) $ is identified with a relative Poincar\'e embedding with underlying map
$F\: P \times D^j \to M \times D^j$, where $j$ is sufficiently large, and $F$ is equipped with a homotopy to $f\times \text{id}$
relative to $\partial_0 P \times D^j$. An  application of the homotopy extension property
allows one to find a concordance to an embedding having  underlying map $ f\times \text{id}$. Without loss of generality,
we may assume that $F = f\times \text{id}$.

\begin{proof}[Proof of Theorem \ref {bigthm:Smale-Hirsch}] 
The absolute case ($\partial_0 P = \emptyset$) of Theorem \ref{bigthm:Smale-Hirsch} is \cite[thm.~A]{Klein_immersion}.
For clarity, we first recall the proof in the absolute case and then explain the modifications needed
in the relative case. Let $P_j = P \times D^j$ and $M_j = M \times D^j$. Then $F \in F(P_j,M_j)$.
In the absolute case, recall that $\partial_0 P = \emptyset$ and $\partial_1 P = \partial P$.

The easy direction is when a Poincar\'e immersion is given. 
Then for $j$ sufficiently large, $F: P_j \to M_j$ underlies a codimension-zero Poincar\'e embedding, and
\[
\nu_{P_j} \simeq F^\ast \nu_{M_j}
\]
(relative to $\partial_0 P_j$)
 by Lemma \ref{lem:inherit}. Then take the pullback along the inclusions $P  \subset P_j$ and $M \subset M_j$
 to conclude that $\nu_P \simeq f^\ast \nu_M$.

Conversely, assume that $\nu_P \simeq f^\ast \nu$, where $\nu = \nu_M$. Let $f_{j}$ be the composition
\[
P @> \subset >> P_j@> F>> M_j \, .
\]
If $j$ is large, then by \cite[thm.~A]{Klein_dualizing2} (or \cite[\S3]{Klein_haef}),  $f_j$ is the underlying map of an embedded thickening of $P$, i.e., there is a 
commutative homotopy pushout square of spaces
\[
\xymatrix{
A \ar[r] \ar[d] & C\ar[d] \\
P \ar[r]_{f_j} & M_j
}
\]
and a factorization of the inclusion $\partial M_j \subset M_j$ as
$\partial M_j @>>>  C @>>> M_j$. At the expense of replacing the map $A \to P$ with a mapping cylinder, we may assume 
that $A\to P$ is an inclusion.  Then the diagram
defines a Poincar\'e embedding of the Poincar\'e pair $(P,A)$ in the interior of  $M_j$. 
The diagram induces a degree one umkehr map of Thom spaces
\[
\Sigma^j M^\nu/(\partial M)^\nu @>f_j^! >> P^{f^\ast\nu}/A^{f^\ast \nu} 
\]
and Lemma \ref{lem:inherit}  shows that the map $A\to P$ defines a $(d+j)$-dimensional thickening in the sense of
\S\ref{subsec:alternative-thickening}, where the Spivak fibration of $(P,A)$ is identified with $f^\ast \nu$.

Consequently, the maps $\partial P_j \to P$ and $A \to P$ define $(d+j)$-dimensional Poincar\'e thickenings of $P$ (where
$\partial P_j = S^j_P \partial P$) whose Spivak fibrations coincide. Then application of \cite[lem.~11.2]{klein2025poincaresurgery}
shows that the two thickenings are concordant when $j$ is sufficiently large. A choice of concordance enables us to replace
$A\to P$ in the above square with the inclusion $\partial P_j \to P_j$ and we obtain in this way 
an element of $I_f(P,M)$ (in the absolute case).

We now explain the necessary modifications to the above argument to handle the relative case. The proof is without change in the easy direction and the details of that are left
to the reader. 

Conversely, assume that there is a fiber homotopy equivalence $\nu_P \simeq f^\ast\nu$ rel $\partial_0 P$. Then we proceed as above (with relative embeddings/thickenings in place of embeddings/thickenings)
by replacing \cite[thm.~A]{Klein_dualizing2} (or \cite[\S3]{Klein_haef}) with \cite[thm.~3.2]{Klein_haef2}.\footnote{The argument of   \cite[thm.~3.2]{Klein_haef2} is not manifold-free. Using
the relative dualizing spectrum, there exists a homotopy-theoretic proof which parallels \cite[thm.~A]{Klein_dualizing2}.}   Then substitute Proposition \ref{prop:relL} in place of
 \cite[lem.~11.2]{klein2025poincaresurgery}.
\end{proof}

The above proof achieves a bit more.
Let 
\[
\cal H(\nu_P,f^\ast\nu)
\]
denote the set of homotopy classes of fiber homotopy equivalences $\nu_P @> \sim >> f^\ast \nu$ which restrict to the identity
along $\partial_0 P$. Here is an explicit description: The stable spherical fibrations may be viewed as a pair of maps $\nu_P,f^\ast\nu \: P \to BG$, where $BG$ is the classifying space of stable spherical fibrations.
As these are required to agree on $\partial_0 P$, we obtain a fixed map $S_P \partial_0 P = P \times \partial I \cup \partial_0 P \times I \to BG$.
Then we may define $\cal H(\nu_P,f^\ast\nu)$ as the set of homotopy classes
\[
[P \times I,BG\rel S_P \partial_0 P]\, .
\]
The Pontryagin-Thom construction in the proof and the uniqueness of the 
Spivak fibration produce a function
\begin{equation}\label{eqn:bijection-I-H}
\pi_0(I_f(P,M)) \to \cal H(\nu_P,f^\ast\nu)\, .
\end{equation}

\begin{add} \label{add:rel-Smale-Hirsch} The function \eqref{eqn:bijection-I-H}
is a bijection.
\end{add}

\begin{rem} Note that the cofiber of $S_P \partial_0 P\to P\times I$ is the reduced suspension $\Sigma (P/\partial_0 P)$.
Therefore, if $[P \times I,BG\rel S_P \partial_0 P]$ is non-empty, it is equipped with a free and transitive action of the abelian group
$[\Sigma (P/\partial_0 P),BG]_\ast = [P/\partial_0 P,G]_\ast$. 

We conclude that if  $I_f(P,M)$ is non-empty, then there is  a (non-canonical) bijection
\[
\pi_0(I_f(P,M))  \cong [P/\partial_0 P,G]_\ast\, .
\]
The smooth analogue of this statement is well known.
\end{rem}

\section{Proof of Theorems \ref{bigthm:unstable-normal} and \ref{bigthm:Whitney}} \label{sec:unstable-normal}

\begin{proof}[Proof of Theorem \ref{bigthm:unstable-normal}]   
The proof is essentially the same as that of \cite[thm.~B]{Klein_compression}.
In what follows we implement the notation $X_S = X\times S$ appearing in the proof of Theorem \ref{bigthm:compression}.
 It is clear that if the immersion $f\: P \to M$ compresses, then $f^!$ is in the image of $E$.
Conversely, suppose that 
 \[
 \alpha\in [M/\!\!/\partial_1 M,P/\!\!/\partial_1 P; \rel \partial_0 M/\!\!/\partial_{01}M ]_{R(M)}
 \] 
 satisfies $E(\alpha) = f^!$. We will explain how $f$ and $\alpha$ determine an embedding $P_J \to M_I$ relative
 to $(\partial_0 P)_J$ and 
 having a trivial link. Assuming this to be the case,  the first part of the theorem will follow from Theorem \ref{bigthm:compression}.
 
 We first construct the embedding $P_J \to M_I$. We will produce an object $W$ in
 the left fiber of the functor
 \[
 \kappa^{P_J}\: \cal I^h((\partial_1 P_J)\cup_{(\partial_{01} P_J)} (\partial_1 M_I)) \to \cal I^h((\partial M_I)) \, .
 \]
We set
 \[
 W := P/\!\! /\partial_1 P = P \cup_{\partial_1 P} M\ ,
 \]
where the structure map
 \[
 \partial_1 (P_J) \cup_{\partial_{01} (P_J)} \partial_1(M_I) \to P/\!\!/\partial_1 P 
 \]
 is defined as follows: On the summand $\partial_1 (P_J)  := S_P \partial_1 P$, we
 take it to be the map $S_P \partial_1 P \to P/\!\!/\partial_1 P$ which collapses the left summand $P \subset S_P \partial_1 P$ to $M$, i.e.,
 it is the evident map of double mapping cylinders
 \[
 P_0 \cup (\partial_1 P)_I \cup P_1  @>>>  M \cup (\partial_1 P)_I \cup P_1 \, .
 \]
 On the other summand $\partial_1 (M_I) = M/\!\!/\partial_1 M$, use any representative of the homotopy class
  $\alpha\in [M/\!\!/\partial_1 M,P/\!\!/\partial_1 P; \rel \partial_0 M/\!\!/\partial_{01} M]_{R(M)}$. Then the maps on each summand
  automatically agree on  the overlap $S_{\partial_0 P} \partial_{01} P =  S_{\partial_0 M} \partial_{01} M$, since
$\alpha$ is fixed on $\partial_0 M/\!\!/\partial_{01} M$. 

Next, we observe that the commutative square of pairs\footnote{The upper right corner of the square is given by $\alpha$; 
here we have ignored the fact that $\alpha$ is not an inclusion. To repair this, one
should replace the pairs in the square with mapping cylinder pairs.}
\[
\xymatrix{
(\partial_1 (P_J),S_{\partial_0 P} \partial_{01} P) \ar[r] \ar[d] & (P/\!\!/\partial_1 P,M/\!\!/ \partial_1 M) \ar[d] \\
(P_J, (\partial_0 P)_J) \ar[r] & (M_I,\partial (M_I))
}
\]
is a homotopy pushout. We infer that gluing in $P_J$ yields a weak equivalence  $P/\!\!/\partial_1 P \cup_{\partial_1 P_J}  P_J \simeq M_I$
which is standard when restricted to
\[
(\partial_0 P)_J \cup M/\!\!/ \partial_1 M = (\partial_0 M)_I \cup  S_M \partial_1 M = \partial (M_I)\, .
 \]
 Set $\partial W = \partial_1 P_J \cup \partial_1 M_I = \partial_1 P_J \cup M/\!\!/\partial_1 M$.
Then Poincar\'e duality for the pair $(W,\partial W)$ is a straightforward consequence of a relative version of 
\cite[lem.~2.3]{Klein_haef};  the details are left to the reader. This completes the construction of
the embedding $P_J \to M_I$.
 
 By construction, the associated pairwise link
\[
(P,\partial_0 P) \to (P/\!\!/\partial_1 P,\partial_0 P/\!\!/\partial_{01} P)
\]
factors through $(M,M)$ and is consequently 
 trivial. As the pairwise link determines the link $\lambda \in [P/\!\!/\partial_0 P, P/\!\!/\partial_1 P]_{R(M)}$, the latter is trivial.
 By Theorem \ref{bigthm:compression}, up to concordance, the relative embedding $P_J \to M_I$ compresses to a Poincar\'e embedding $P \to M$.
 This completes the first part of the theorem.
 
 To complete the proof, we need to show that the homotopy class of the collapse map $f^!$ of the constructed embedding $f\: P \to M$  coincides with $\alpha$.
To accomplish this we transfer the argument appearing in the proof of \cite[thm.~B]{Klein_compression} to the relative setting. The convention for the rest of the proof
will be that {\it maps are assumed  fiberwise over $M$ and are relative to}  $\partial_0 M/\!\!/\partial_{01} M = \partial_0 P/\!\!/\partial_{01} P$.

Recall  the pair $(C,A)$ appearing in the proof of Theorem \ref{bigthm:compression}.
 By the construction of the embedding, $f^!$ is identified with the map
 \[
 M/\!\!/\partial_1 M \simeq   M_0 \cup_{\partial_1 M} (A \cup_{\partial_1 P} P) \to  \bar M  \cup_{\partial_1 P} P \simeq  P/\!\! /\partial_1 P  \, ,
 \]
  where the restriction to $M_0$ is given by the inclusion $M_0 \to \bar M$
 and the restriction to $A \cup_{\partial_1 P} P$ is given by the map $A\to \bar M$ amalgamated along $\partial_1 P$ with the identity map of $P$.
 
Using Figure \ref{fig:to-be-compressed-in-proof}, we rewrite $f^!$ as follows: Let 
\[
P' = (\partial_1 P)_J \cup_{\partial_1 P_{2/3}}  P_{2/3} \qquad \text{ and } \qquad M' = A\cup_{\partial_1 P_{1/3}} P'\, .
\]
Then $f^!$ coincides with the fiberwise homotopy class of the evident map
\[
 M/\!\!/\partial_1 M \simeq    M_0 \cup_{\partial_1 M} M' \to \bar M \cup_{\partial_1 P_{1/3}} P' \simeq  W\, .
\]
Moreover, there is an evident factorization of the latter as
\[
 M_0 \cup_{\partial_1 M} M'  \to M_0 \cup_{\partial_1 M} C \to W\, .
 \] 

On the other hand, the class $\alpha$ is represented by the composition
 \[
 M_0 \cup_{\partial_1 M} M_1 \to M_0 \cup_{\partial_1 M} C \to W\, .
 \]
Summarizing, the  restriction of the map $M_0 \cup_{\partial_1 M} C \to W$  to  $M_0 \cup_{\partial_1 M} M'$ induces $f^!$
and its  restriction to 
$M_0 \cup_{\partial_1 M} M_1$ induces  $\alpha$:
\begin{equation} \label{eqn:alpha-f-!}
\xymatrix{
M_0 \cup_{\partial_1 M} M' \ar[dr]  \ar@/^1.5pc/[rrd]^(.6){f^!} \\
& M_0 \cup_{\partial_1 M} C \ar[r] & W \\
M_0 \cup_{\partial_1 M} M_1 \ar[ur] \ar@/_1.5pc/[rru]_(.6){\alpha}
}
\end{equation}
As in the proof of Theorem \ref{bigthm:compression}, the maps $M' \to C$ and $M_1 \to C$ are weak equivalences, and therefore the diagonal maps
in \eqref{eqn:alpha-f-!} are weak equivalences. We conclude that given an unstable class $\alpha$ with $E(\alpha) = f^!$, the embedding
can be chosen so that its unstable collapse is $\alpha$.
\end{proof}
 
\begin{proof}[Proof of Theorem \ref{bigthm:Whitney}] Since $p \le d-3$ and $3p+4 \le 2d$, 
by Theorem \ref{bigthm:unstable-normal} it suffices to show that the function
\[
\mbox{\small 
$E\: [M/\!\!/\partial_1 M,P/\!\!/\partial_1 P; \rel \partial_0 M/\!\!/\partial_{01}M]_{R(M)} \to \{M/\!\!/\partial_1 M,P/\!\!/\partial_1 P; \rel \partial_0 M/\!\!/\partial_{01}M\}_{R(M)} 
$}
\]
is surjective. Since $(P;\partial_0 P,\partial_1 P)$ has homotopy spine dimension $\le p$,
the object $P/\!\!/\partial_1 P\in R(M)$ is $(d-p-1)$-connected by the Blakers-Massey theorem. 
Moreover, $\dim_{\partial_0 M/\!\!/\partial_{01}M} M/\!\!/\partial_1 M \le d$.
Then the result follows from the first part of Lemma \ref{lem:fringe}, since
$d \le 2(d-p-1) + 1$ is equivalent to $2p+1 \le d$.
\end{proof}

\section{Proof of Theorem \ref{bigthm:handle}} \label{sec:handle}

Let $[\alpha] \in \pi_k(M,\partial M)$ be as in Theorem \ref{bigthm:handle}.

\begin{lem} \label{lem:immersion-handle} Assume $k \ge 3$. Then there exists
a Poincar\'e triad $(N;\partial_0 N,\partial_1 N)$ with $\partial_0 N = \partial_0 H$
and a weak equivalence $(N,\partial N) @>{}_\sim >> (M,\partial M)$. 

Furthermore, with respect to the induced isomorphism $\pi_k(N,\partial N) \cong  \pi_k(M,\partial M)$,
there is an immersion
 \[
f\: H \to N
 \] 
 relative to $\partial_0 H$ and
 representing $[\alpha]$.
 \end{lem}
 
 \begin{proof} Represent $[\alpha]$ as a map of pairs
\[
\alpha\: (D^k,S^{k-1}) \to (M,\partial M)\, .
\]
By \cite[thm.~F]{Klein_compression}, the map $S^{k-1} \to \partial M$ admits an embedded thickening when $k \ge 3$. 
Since $S^{k-1}$ is a Poincar\'e duality space of dimension $k-1$, the embedded thickening is a codimension $k$ Poincar\'e embedding of $\alpha\: S^{k-1} \to \partial M$ so there is
a homotopy pushout
\[
\xymatrix{
E \ar[r] \ar[d]_{\xi} & C \ar[d] \\
S^{k-1} \ar[r]_{\alpha} & \partial M\, ,
}
\]
where $\xi\: E\to S^{k-1}$ is a $(k-1)$-spherical fibration \cite[thm.~I.4.1]{Browder}, \cite[thm.~B]{Klein_fibration}.
Furthermore, $\xi\: E \to S^{k-1}$ satisfies the following identity 
\[
\xi + \alpha^\ast\nu_{\partial M} = \nu_{S^{k-1}} = 0\, ,
\]
in the reduced Grothendieck group $[S^{k-1},BG]$, where $\nu_{\partial M}$ is the Spivak normal fibration of $\partial M$ by Lemma \ref{lem:whitney-sum-formula} below (here we used the fact  that $(D(\xi),S(\xi))$ is a Poincar\'e thickening of $S^{k-1}$).  
Consequently, as a stable spherical fibration the homotopy class of $\xi$ coincides with $\nu_{S^{k-1}} -  \alpha^\ast\nu_{\partial M}$, where $\nu_{S^{k-1}}$ is trivial.
Moreover, 
$\nu_{\partial M} = i^*\nu_{M}$, where $i\: \partial M \to M$ is the inclusion, and  $\alpha^\ast\nu_M\: D^k \to BG$ is a trivialization of $\alpha^\ast\nu_{\partial M}$.
It follows that the composition
\[
S^{k-1} @> \xi >>  BG_k \to BG
\]
comes equipped with a trivialization, where in the display, $BG_k$ is the classifying space of $(k-1)$-spherical fibrations. 
But the map $BG_k \to BG$ is $k$-connected, so the trivialization lifts to a map $D^k \to BG_k$ extending $\xi$.
Therefore, $\xi$ is homotopically trivial and the Poincar\'e embedding $S^{k-1} \to \partial M$ extends to a codimension-zero Poincar\'e
embedding $\phi_0\: \partial_0 H \to \partial M$ with homotopy pushout diagram
\[
\xymatrix{
\partial_{01} H \ar[r] \ar[d] & C \ar[d] \\
\partial_0 H \ar[r] & \partial M\, .
}
\]
Let $\partial_1 N$ be the mapping cylinder of the map $\partial_{01} H \to C$ and set 
\[
\partial N := \partial_0 H \cup_{\partial_{01} H}  \partial_1 N\, .
\]
Then $(\partial N,\partial_0 N,\partial_1 N)$ is a codimension-one splitting with $\partial_0 N = \partial_0 H$
and there is a homotopy equivalence $\partial N \simeq \partial M$. So up to
weak equivalence, we have constructed a codimension-one splitting of $\partial M$. Let
$N$ be the mapping cylinder of the composition 
\[
\partial N @> \simeq >> \partial M @> \subset >> M\, .
\]
Then $(N;\partial_0 N,\partial_1 N)$ is a Poincar\'e triad and
there is a weak equivalence 
\[
(N,\partial N) @> {}_\sim >> (M,\partial M)\, .
\]
Without loss of generality, for the rest of the proof we will make the identification $(N,\partial N)  = (M,\partial M)$.
With respect to this, we have a Poincar\'e triad decomposition $(M;\partial_0 M,\partial_1 M)$ with $\partial_0 M = \partial_0 H$.
Consequently, $[\alpha]$ is represented by a map of pairs
\[
(f,\phi_0) \: (H,\partial_0 H) \to (M,\partial M)\, 
\]
in which $\phi_0\: \partial_0 H = \partial_0 M \to \partial M$ is the inclusion.

By the way the trivialization of $\xi$ was defined, we have a fiber homotopy equivalence of pairs
$(\nu_H,\nu_{\partial_0 H}) \simeq (f^\ast \nu_M,{\phi_0^\ast}\nu_{\partial M})$.
 Using Theorem \ref {bigthm:Smale-Hirsch}, we infer that $\alpha$ is represented by 
an immersion.
\end{proof}

\begin{proof}[Proof of Theorem \ref{bigthm:handle}]
By Lemma \ref{lem:immersion-handle}, there is a Poincar\'e triad
$(N;\partial_0 N,\partial_1 N)$ with $\partial_0 N = \partial_0 H$, 
a weak equivalence $(N,\partial N) @>{}_\sim >> (M,\partial M)$ and an immersion
\[
f\: H \to N
\]
$\rel \partial_0 H$ representing $[\alpha]$.  As in the proof
of Lemma  \ref{lem:immersion-handle}, we may assume without loss of generality that $N = M$,
where $(M;\partial_0 M,\partial_1 M)$ is a Poincar\'e triad, $\partial_0 M = \partial_0 H$,
and the restriction of $f$ to $\partial_0 H$ is the inclusion. Then we have a map
of pairs
\[
(f,\phi_0) \: (H,\partial_0 H)\to (M,\partial M)
\]
which represents $[\alpha]$, where $\phi_0$ coincides with the inclusion $\partial_0 M \subset \partial M$.

We therefore have an associated $\pi$-equivariant stable collapse map of pairs
 \[
 (f^!,\phi_0^!) \: (\tilde M_+,\partial \tilde M_+) \to (\tilde H/\partial_1 \tilde H, \partial_0 \tilde H/\partial_{01} \tilde H)\, ,
 \]
 where $\pi = \pi_1(M)$, $\tilde M \to M$ is the universal cover, and  if $Y \to M$ is a map, we set $\tilde Y := Y \times_M \tilde M$.
 As $(M,\partial M)$ is $2$-connected, the map $\partial M \to M$ induces an isomorphism on fundamental groups.
 In particular, $\partial \tilde M \to \partial M$ is the universal cover.
 
 The map $\phi_0^!$  exists unstably. Therefore,  we have
 \[
 [f^!] \in \{\tilde M_+,\tilde H/\partial_1 \tilde H; \partial \tilde M_+ \rel \partial_1 \tilde M_+\}_{T_\ast(\pi)} \, .
 \]
 Note that the above $\pi$-equivariant stable collapse is the image of the fiberwise stable collapse from the introduction (also denoted by $f^!$)
 with respect to the function
  \[
({-})^\sharp \:  \{M^+, H/\!\!/\partial_1 H; \partial M^+ \rel \partial_1 M^+\}_{R(M)} \to  \{\tilde M_+,\tilde H/\partial_1 \tilde H;
\partial \tilde M_+ \rel \partial_1 \tilde M_+\}_{T_{\ast}(\pi)}
 \]
 of \S\ref{subsec:retract-desuspend}. We also observe that there is a preferred equivariant weak equivalence
 \[
 \tilde H/\partial_1 \tilde H \simeq S^k \smsh (\pi_+) \, .
 \]
 

By Theorem \ref{bigthm:unstable-normal} and Corollary \ref{cor:fringe}, the obstruction to compressing the immersion 
 $(f,\phi_0)$ to an embedding  is given by the relative Hopf invariant 
  \[
 H_{\partial \tilde M_+}(f^!) \in \{\tilde M/\partial \tilde M,D_2(S^k \smsh (\pi_+))\}_{T_\ast(\pi)}
 \]
 of \eqref{eqn:rel-Hopf} (the assumption $k \ge 4$ arises from application of Theorem \ref{bigthm:unstable-normal}).
 The strategy is then to show that there exists an immersion 
 \[
 (\phi;\phi_0,\phi_1)\: (\partial_0 H \times I ;\partial_0 H \times \{0\}, \partial_0 H \times \{1\})
 \to (\partial M \times I; \partial M \times \{0\},\partial M \times \{1\})\, ,
 \]
 such that
 \begin{enumerate}[(i).]
 \item $\phi_1$ is an embedding, and
 \item the immersion
  \[
 f\cup\phi\: H \cup_{\partial_0 H \times \{0\} } (\partial_0 H \times I) \to M \cup_{\partial M} (\partial M \times I)\, ,
 \]
 obtained by gluing $\phi$ to $f$, compresses to an embedding (the immersion $f\cup \phi$ is depicted in Figure \ref{fig:stacking}).
 \end{enumerate}
 Then the resulting compression
is an embedding representing $\alpha$.
 
 
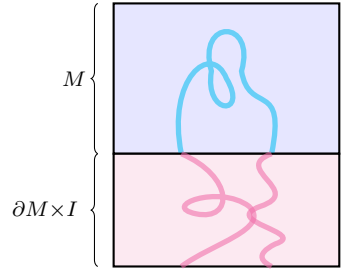
\begin{figure}
\centering
\begin{tikzpicture}[scale=.5]

\fill[blue!10] (-1,0) rectangle (5,4);
\draw[thick] (-1,0) rectangle (5,4);

\draw[decorate,decoration={brace, raise=6pt}]
  (-1,0) -- (-1,4)
  node[midway,xshift=-15pt] {\tiny $M$};
 
\draw[decorate,decoration={brace,raise=6pt}]
  (-1,-3) -- (-1,0)
  node[midway,xshift=-26pt] {\tiny $\partial M {\times} I$};

\draw[line width=2pt, white!30!cyan, draw opacity=0.6, smooth]
  (0.8,0)
  .. controls (0.4,2) and (1.6,3) ..
  (2,2)
  .. controls (2.6,1) and (1.4,1) ..
  (1.6,2.2)
  .. controls (1.2,3.6) and (2.8,3.6) ..
  (2.4,2.2)
  .. controls (2.6,1) and (3.6,2) ..
  (3.2,0);

\fill[magenta!10] (-1,0) rectangle (5,-3);
\draw[thick] (-1,0) rectangle (5,-3);

\draw[line width=2pt, white!30!magenta, draw opacity=0.6, smooth]
  (0.8,0)
  .. controls (2.2,-0.6) and (2.2,-1.8) ..
  (1.2,-1.6)
  .. controls (0.2,-1.4) and (1.8,-0.8) ..
  (2.6,-1.4)
  .. controls (3.2,-2.0) and (1.6,-2.4) ..
  (0.8,-3);

\draw[line width=2pt, white!30!magenta, draw opacity=0.6, smooth]
  (3.2,0)
  .. controls (2.1,-0.4) and (4.1,-1.0) ..
  (2.9,-1.4)
  .. controls (1.8,-1.8) and (4.2,-2.1) ..
  (3.3,-2.4)
  .. controls (2.7,-2.7) and (3.0,-2.9) ..
  (3.2,-3);

\end{tikzpicture}
\caption{A smooth manifold immersion that schematically depicts the Poincar\'e immersion $f\cup \phi$ appearing
in the proof of Theorem \ref{bigthm:handle}. The bottom rectangle 
is an immersion $\phi\: \partial_0 H {\times} I \to \partial M {\times} I$ which restricts to an embedding $\partial_0 H {\times} \partial I \to \partial M {\times} \partial I$. 
The top rectangle is  an immersion $f\: H \to M$. The immersions coincide at the embedding $\phi_0\: \partial_0 H {\times} \{0\} \to \partial M {\times} \{0\}$.
In the figure, $\partial M \times \{0\}$ is the edge where the rectangles meet and  $\partial M \times \{1\}$ is the bottom edge of the lower rectangle.}
\label{fig:stacking}
\end{figure}
 
The  argument will make use of the commutative square of spaces
\[
 \xymatrix{
E_0(\partial_0 H,\partial M) \ar[r] \ar[d]  & F_{T_\ast(\pi)}(\partial \tilde M_+,\partial_0 \tilde H/\partial_{01}\tilde H)\ar[d]\\
I_0(\partial_0 H,\partial M) \ar[r] & F^{\text{st}}_{T_\ast(\pi)}(\partial \tilde M_+,\partial_0 \tilde H/\partial_{01}\tilde H)
}
 \]
 in which 
 \begin{enumerate}[(i).] 
 \item $E_0(\partial_0 H,\partial M)$ is the homotopy fiber of the map $E(\partial_0 H,\partial M) \to F(\partial_0 H,\partial M)$ 
 at the underlying map of $\phi_0$,
 \item $I_0(\partial_0 H,\partial M)$ is the homotopy fiber of the map $I(\partial_0 H,\partial M) \to F(\partial_0 H,\partial M)$ 
 at the underlying map of $\phi_0$,
 \item $F_{T_\ast(\pi)}(\partial \tilde M_+,\partial_0 \tilde H/\partial_{01}\tilde H)$ is the space of equivariant based maps $\partial \tilde M_+ \to \partial_0 \tilde H/\partial_{01}\tilde H$,
 \item $F_{T_\ast(\pi)}^{\text{st}}(\partial \tilde M_+,\partial_0 \tilde H/\partial_{01}\tilde H)$ is the space of equivariant stable maps $\partial \tilde M_+ \to \partial_0 \tilde H/\partial_{01}\tilde H$, 
 \item the vertical maps are evident, and 
 \item the horizontal maps are given by the equivariant collapse.
  \end{enumerate}
  (Note: there is no need to derive the function spaces as all objects are fibrant and the domains are cofibrant by assumption.)

 By combining \cite[thm.~B]{Klein_compression}, Remark \ref{rem:unstable-normal}, and Corollary \ref{cor:fringe}, 
 the square is $0$-Cartesian if $k \ge 4$. By definition, the homotopy fiber of the left vertical map of the square
 taken at $\phi_0$ is identified with the space of immersions $ (\phi;\phi_0,\phi_1)$ as above. Denote this homotopy fiber by
 \[
EI_0(\partial_0 H,\partial M) \, .
 \]
The next step will be to modify the right vertical map of the square to obtain another $0$-Cartesian square.
 Since the map 
 $\partial_0 \tilde H/\partial_{01} \tilde H \to \tilde H/\partial_{1} \tilde H$ is $(2k-1)$-connected, if we replace 
 $\partial_0 \tilde H/\partial_{01} \tilde H$ by $\tilde H/\partial_{1} \tilde H \simeq S^k \smsh (\pi_+)$ in the square, the resulting square
\[
 \xymatrix{
E_0(\partial_0 H,\partial M) \ar[r] \ar[d]  & F_{T_\ast(\pi)}(\partial \tilde M_+,S^k \smsh (\pi_+))\ar[d]\\
I_0(\partial_0 H,\partial M) \ar[r] & F^{\text{st}}_{T_\ast(\pi)} (\partial \tilde M_+,S^k \smsh (\pi_+))
}
 \]
 is still $0$-Cartesian by the Freudenthal suspension theorem. Next, we use \eqref{eqn:EHP} to obtain a $(k-1)$-Cartesian square of function spaces
 \[
  \xymatrix{
 F_{T_\ast(\pi)}(\partial \tilde M_+,S^k \smsh (\pi_+)) \ar[r] \ar[d] &  F^{\st} _{T_\ast(\pi)}(\partial \tilde M_+,C(S^k \smsh (\pi_+)))\ar[d] \\
  F^{\text{st}}_{T_\ast(\pi)}(\partial \tilde M_+,S^k \smsh (\pi_+)) \ar[r] & F^{\st}_{T_\ast(\pi)}(\partial \tilde M_+,D_2(S^k \smsh (\pi_+)))
 }
 \]
 The homotopy fiber of the right vertical map of the latter square at the constant map is identified with $F^{\st}_{T_\ast(\pi)}(\Sigma \partial \tilde M_+,D_2(S^k \smsh (\pi_+)))$.
 
By assembling the above, we obtain a $0$-Cartesian square
\[
 \xymatrix{
E_0(\partial_0 H,\partial M) \ar[r] \ar[d]  & F_{T_\ast(\pi)} (\partial \tilde M_+,C(S^k \smsh (\pi_+)))\ar[d]\\
I_0(\partial_0 H,\partial M) \ar[r] & F^{\st}_{T_\ast(\pi)}(\partial \tilde M_+,D_2(S^k \smsh (\pi_+)))\, .
}
 \]
Using this last square, we take vertical homotopy fibers at the basepoint and then path components. This results in a surjection 
\begin{equation} \label{eqn:map-of-fibers}
   \psi\:  \pi_0(EI_{0}(\partial_0 H,\partial M))  \to \{\Sigma \partial \tilde M_+,D_2(S^k \smsh (\pi_+))\}_{T_\ast(\pi)}
 \end{equation}
 when $k \ge 4$. In fact, the function \eqref{eqn:map-of-fibers} has the following concrete description: It coincides with the composition
 \[
\xymatrix{
\pi_0(EI_{0}(\partial_0 H,\partial M))  \ar[r] &
\{(\partial \tilde M \times I)_+,S^k \smsh (\pi_+);(\partial \tilde M \times \partial I)_+\rel \, (\partial \tilde M \times\{0\})_+\}_{T_\ast(\pi)} \ar[d]^{H_{(\partial \tilde M \times \partial I)_+}}\\
&  \{\Sigma (\partial \tilde M_+),D_2(S^k \smsh (\pi_+))\}_{T_{\ast}(\pi)}
}
\]
where the horizontal map is the composition of the equivariant collapse of an immersion $(\phi,\phi_0,\phi_1)$ 
 with the map $\partial_0 \tilde H/\partial_{01} \tilde H \to \tilde H/\partial_1 \tilde H \simeq S^k \smsh (\pi_+)$
  and the vertical map is given by the relative Hopf invariant.
 
 Let 
 \[
 \delta^\ast\: \{\Sigma (\partial\tilde M_+), D_2(S^k \smsh (\pi_+))\}_{T_\ast(\pi)} \to \{\tilde M/\partial \tilde M, D_2(S^k \smsh (\pi_+))\}_{T_\ast(\pi)}
 \]
 be given by $\alpha \mapsto \alpha\circ \delta$, where $\delta$ is the Barratt-Puppe connecting map of the equivariant cofiber sequence $\partial \tilde M_+ \to \tilde M_+ \to \tilde M/\partial \tilde M$. Since $M$ has the homotopy type of a CW complex of dimension $\le 2k-1$ and $D_2(S^k \smsh (\pi_+))$ is $(2k-1)$-connected,  it follows that $\{\tilde M_+, D_2(S^k \smsh (\pi_+))\}_{T_\ast(\pi)}$ is trivial.
 Hence,
$\delta^\ast$ is surjective by exactness. We infer that the composition
\[
\delta^\ast \circ \psi \: \pi_0(EI_{0}(\partial_0 H,\partial M))  \to \{\tilde M/\partial \tilde M, D_2(S^k \smsh (\pi_+))\}_{T_\ast(\pi)}
\]
is also surjective. Let $(\phi,\phi_0,\phi_1) \in EI_{0}(\partial_0 H,\partial M)$ be an element such that
$\delta^\ast \circ \psi(\phi,\phi_0,\phi_1) = -H_{\partial \tilde M_+}(f^!)$.

Then 
\begin{align*}
H_{\partial \tilde M_+}((f\cup_{\phi_0} \phi)^!) &= H_{(\partial \tilde M \times \{1\})_+ }(f^! \cup_{\phi_0^!} \phi^!)\, , \\
&= H_{\partial \tilde M_+}(f^!) + H_{(\partial \tilde M \times \partial I)_+}(\phi^!)\circ \delta  \quad \text{ by Lemma } \ref{lem:gluing},\\
&= H_{\partial \tilde M_+}(f^!) + \delta^\ast \circ \psi(\phi,\phi_0,\phi_1)\, , \\
&= 0\, ,
\end{align*}
where  in the first line we have used the evident identification $(f\cup_{\phi_0} \phi)^! = f^! \cup_{\phi_0^!} \phi^!$.
The result then follows by applying Lemma \ref{lem:desuspension}, Lemma \ref{lem:fringe},
and Theorem \ref{bigthm:unstable-normal} to the immersion $f\cup_{\phi_0} \phi$ 
with $\Theta(f^!,\phi^!) = (f\cup_{\phi_0} \phi)^!$.
\end{proof} 

\appendix

\section{Relative Poincar\'e thickenings}  \label{sec:thickenings}

\subsection{Twisted stabilization} Let $K$ be a finite CW complex and let $\xi\: S(\xi) \to K$ have the fiber homotopy type of a $(j-1)$-spherical fibration.
If  $p\: E\to K$ is a map, then the  {\it fiberwise join of $p$ with $\xi$} is
\begin{align*}
S^\xi_K E &:= \hocolim (E @<<< E\times_K S(\xi) @>>> S(\xi)) \\
&\simeq D(p^\ast \xi) \cup_{S(p^*\xi)} S(\xi)\, ,
\end{align*}
where $p^\ast \xi$ is the pullback of $\xi$ to $E$  and $D(p^\ast \xi)$ is the mapping cylinder of $p^\ast \xi$ (cf.~\cite[\S3]{Klein_immersion}).
The construction defines a functor
\[
S^\xi\: T_{/K} \to T_{/K}
\]
called {\it stabilization by $\xi$}. (In the case of a trivial fibration with $j = 1$, the functor is just unreduced fiberwise suspension $S_K$.)

There is an evident morphism 
\[
S^\xi_K E \to S^\xi_K K = D(\xi)\, ,
\]
which is an inclusion when $E \to K$ is a cofibration.

If $\xi'$ is a $(j'-1)$-spherical fibration over $K$, then there is a weak equivalence
\[
S^{\xi+\xi'}_K E \simeq S^{\xi'}_K S^{\xi}_K E \, .
\]
Here, $\xi+\xi'$ denotes the fiberwise join of $\xi$ and $\xi'$; it has the fiber homotopy type of 
a $(j+j'-1)$-spherical fibration. 

Note that if $\xi $ is fiber homotopically trivial, then there is a weak equivalence of spaces over $K$
\begin{equation} \label{eqn:sum-stable}
S^{\xi}_K E \,\, \simeq \,\, E\times D^{j} \cup_{E\times S^{j-1}} K \times S^{j-1}\, .
\end{equation}

Suppose that $(K,\partial K)$ is a Poincar\'e pair of dimension $d$. Then 
\[
(D(\xi), S^\xi_K \partial K)
\]
is a Poincar\'e pair of dimension $d+j$.

\begin{lem}\label{lem:whitney-sum-formula} Let $\nu_K$ and $\nu_{D(\xi)_{|K}}$ be the Spivak normal fibrations
of $(K,\partial K)$ and $(D(\xi), S^\xi_K \partial K)$ respectively. Then 
\[
\xi + \nu_{D(\xi)_{|K}}= \nu_K 
\]
in the Grothendieck group $[K,BG]$.
\end{lem}

\begin{proof} We provide the argument when $\partial K = \emptyset$ as the general case is similar.
Let $i \:K \to D(\xi)$ be the inclusion and  choose a $(q-1)$-spherical fibration $\hat \nu$ representing
the Spivak normal fibration $\nu_{D(\xi)}$ of $(D(\xi),S(\xi))$, where $q$ is sufficiently large. 
Then we have a Poincar\'e pair
$(D(\xi + i^\ast \hat \nu),S(\xi+i^\ast \hat \nu))$ of dimension $d+j+q$.
The normal invariant is then represented by a map
\[
S^{d+j+q} \to D(\xi)^{\hat \nu}/S(\xi)^{\hat \nu}\, ,
\]
where in the display, superscripts denote ordinary Thom spaces.
Under the equivalence
\[
D(\xi)^{\hat \nu} /S(\xi)^{\hat \nu} \simeq K^{\xi + i^\ast \hat \nu}\,   ,
\]
the Hurewicz image of the normal invariant corresponds, via the Thom isomorphism, to a fundamental class of $K$. Consequently,
$\xi + i^\ast \hat \nu$ represents the Spivak fibration of $K$. The result follows.
\end{proof}

\begin{ex} If $\xi$ is trivial, then 
\[
(D(\xi), S^\xi_K \partial K) = (K \times D^j,\partial (K \times D^j))\, .
\]
\end{ex}

\subsection{Poincar\'e thickenings} The notion of a Poincar\'e thickening \cite{Klein_immersion}, \cite[\S12]{klein2025poincaresurgery}
is the Poincar\'e duality space analogue of a manifold thickening in the sense of \cite{Wall_thickening}.

\begin{defn} Suppose that $L$ is a finitely dominated CW complex.
A {\it Poincar\'e thickening} of $L$ of dimension $d$ 
is a pair
\[
((Q,\partial Q), f)
\]
consisting of 
\begin{enumerate}[(i).]
\item a Poincar\'e pair $(Q,\partial Q)$ of dimension $d$,
in which the Poincar\'e triad 
\[
(Q;\emptyset,\partial Q)
\]
has homotopy spine dimension $\le d-3$, and
\item
a weak homotopy equivalence 
\[
f\: L  @> \sim >> Q\, .
\]  
\end{enumerate}
\end{defn}

\begin{rem}  Let $((Q,\partial Q), f)$ be a Poincar\'e thickening of dimension $d$.
If $L$ has the homotopy type of a CW complex of dimension $\ell \le d-3$, then
it follows that $(Q;\emptyset,\partial Q)$ has homotopy spine dimension $\le \ell$. 

To see this, note that it suffices to show that $(Q,\partial Q)$ is $(d-\ell-1)$-connected. From the definition we know that
$(Q,\partial Q)$ is $2$-connected. By Poincar\'e duality, the homology of $(Q,\partial Q)$ with respect to every coefficient bundle
vanishes in degrees $\le d -\ell -1$. By the relative Hurewicz theorem, we infer that $(Q,\partial Q)$ is $(d-\ell -1)$-connected.
\end{rem}

Suppose that $\alpha := ((Q,\partial Q), f)$ and $\beta := ((Q',\partial Q'), f')$ are Poincar\'e thickenings of $L$ of dimension $d$. 
A {\it concordance} from $\alpha$ to $\beta$ 
is a weak homotopy equivalence 
\[
h\: (Q,\partial Q) @> \sim >> (Q',\partial Q')
\]
such that $h \circ f\: L \to Q'$ is homotopic to $f'$. The notion of concordance generates an equivalence relation and
we let 
\[
\cal T_d(L)
\] denote the associated  set of equivalence classes.

According to \cite[\S12]{klein2025poincaresurgery}, $\cal T_\infty(L)$ is non-empty and has a preferred basepoint that is represented by
a {\it stable Euclidean thickening} of $L$. The latter is defined to be any thickening 
\[
e_L := ((Q,\partial Q),g)
\]
of  $L$ of sufficiently large dimension $d$ such that there is a map
$S^d \to Q/\partial Q$ yielding a fundamental class on homology in dimension $d$. Any two such thickenings are stably concordant \cite{Klein_immersion}.\footnote{Note that $e_L$ is the Poincar\'e duality space analogue of selecting a compact framed manifold of dimension $d$ having the homotopy type of $L$. If $d$ is large,
any two such framed manifolds are diffeomorphic.}

If $BG$ is the classifying space of stable spherical fibrations, then
$[L,BG]$ is the abelian group of fiber homotopy classes of
stable spherical fibrations over $L$, where the group operation is given by fiberwise join. The operation of stabilizing a thickening
by a spherical fibration defines
an action
\[
[L,BG] \times \cal T_\infty(L) \to \cal T_\infty(L)
\]
which is free and transitive \cite[cor.~12.3]{klein2025poincaresurgery}. Let
\[
s\: [L,BG] \to \cal T_\infty(L)
\]
be defined by $s([\xi]) = \sigma^\xi (e)$, where $e\in  \cal T_\infty(L)$ is the stable class of a Euclidean thickening $e_L$.
There is also a function
\begin{equation} \label{eqn:thick-to-sphere}
\begin{aligned} 
r\: \cal T_\infty(L) &@>>> [L,BG] \, ,\\
        [Q,f] & \mapsto [f^\ast \nu_Q]\, ,
\end{aligned}
\end{equation}
where $\nu_Q$ is the Spivak normal fibration of the Poincar\'e pair $(Q,\partial Q)$ (compare \cite[prop.~5.1]{Wall_thickening} in the smooth case). 
Then $r(s([\xi])) = -[\xi]$, so both $r$ and $s$ are bijections, and $r^{-1}([\eta]) = s(-[\eta])$.
In particular, the bijection \eqref{eqn:thick-to-sphere} is basepoint preserving, since the Spivak fibration
of a Euclidean thickening is trivial.

\subsection{Relative Poincar\'e thickenings}

\begin{defn} \label{defn:rel-thickening}
Let $(K,L)$ be a finitely dominated cofibrant  pair.
A {\it Poincar\'e thickening} of $(K,L)$ of dimension $d$ 
is a pair
\[
((P;\partial_0 P,\partial_1 P), f)
\]
consisting of a Poincar\'e triad 
\[
(P;\partial_0 P,\partial_1 P)
\]
of dimension $d$ having  homotopy spine dimension $\le d-3$, 
and a weak homotopy equivalence 
\[
f\: (K,L) @> \sim >> (P,\partial_0 P)\, .
\]  
In addition, we require that
the pair $\alpha_L := ((\partial_0 P,\partial_{01} P),f_{|L})$ is a Poincar\'e thickening of dimension $d-1$.
\end{defn}

In this case we say that $((P;\partial_0 P,\partial_1 P), f)$ is a Poincar\'e thickening of $(K,L)$ {\it relative to} $\alpha_L$.
Henceforth, the thickening $\alpha_L$ is held fixed.

\begin{notation} To avoid clutter, we set 
\[
(P,f) := ((P;\partial_0 P,\partial_1 P), f)\, ,
\]
whenever the triad structure on $P$ is understood.
\end{notation}

A pair of  Poincar\'e thickenings $(P, f)$ and $(Q,g)$ of $(K,L)$ relative to $\alpha_L$
are {\it concordant} if there is a weak homotopy equivalence of triads 
\[
h\: (P;\partial_0 P,\partial_1 P) @>\sim >> (Q;\partial_0 Q,\partial_1 Q)
\]
(where $\partial_0 Q = \partial_0 P$)
such that the restriction of $h$ to $\partial_0 P$ is the identity and
  $h\circ f \simeq g \rel L$. Let 
\[
\cal T_d(K,L\rel \alpha_L)
\]
denote the associated set of concordance classes of relative Poincar\'e thickenings. 

\begin{rem} If $L = \emptyset$, we recover $\cal T_d(K)$.
\end{rem} 

\subsubsection{Stabilization}
 The stabilization map
\[
\sigma\: \cal T_d(K,L\rel \alpha) \to \cal T_{d+1}(K,L\rel \sigma(\alpha) ) 
\]
is induced by the operation $ (P;\partial_0 P,\partial_1 P)\mapsto  (P{\times} D^1;(\partial_0 P) {\times} D^1,S_P \partial_1 P)$.
Let
\[
\cal T_{\infty}(K,L,\rel \alpha) = \colim_{j\to\infty}  \cal T_{d+j}(K,L,\rel \sigma^j \alpha)
\]
where the colimit is with respect to $\sigma$. 

\subsubsection{The fiberwise join action} Fix a stable spherical fibration $\xi \: L \to BG$.
Let 
\[
[K,BG \rel \xi]
\]
denote the set of homotopy classes of maps $K \to BG$ relative to $\xi$. An element
of this set is represented by a stable spherical fibration on $K$ which restricts to $\xi$. 

If $[K,BG \rel \xi]$ is non-empty, then fiberwise join  defines a free and transitive action
\[
[K/L,BG] \times [K,BG \rel \xi] \to [K,BG \rel\xi ]\, .
\]
In particular, if $[K,BG \rel \xi]$ is non-empty, then there is a non-canonical bijection
\[
[K,BG \rel \xi]  \cong [K/L,BG] \, .
\]

\subsubsection{Stable classification}

\begin{prop}  \label{prop:relL}  The set $\cal T_\infty(K, L \rel \alpha)$ is non-empty  if and only if the set $[K,BG \rel \nu_\alpha]$ is non-empty.
If this is the case, then  there is a canonical bijection
\[
\cal T_\infty(K,L \rel \alpha) \cong [K,BG \rel \nu_\alpha]\, .
\]
\end{prop}

\begin{rem} If a basepoint of $[K,BG \rel \nu_\alpha]$ is chosen, then one additionally has a bijection
\[
 [K,BG \rel \nu_\alpha] \cong [K/L,BG]_\ast\, .
 \]
\end{rem}

\begin{proof}[Proof of Proposition \ref{prop:relL}](Sketch). If $\alpha$ is a thickening of $L$, we let $\nu_\alpha$ denote its associated Spivak normal fibration.
Then the Spivak normal fibration of a relative thickening defines
a function
 \[
r_\alpha\: \cal T_\infty(K, L \rel \alpha) \to  [K,BG \rel \nu_\alpha]\, .
 \]
 Consequently, if the domain is non-empty, then the codomain is also non-empty.
 
 Conversely, assume that  $[K,BG \rel \nu_\alpha]$ is non-empty. 
 Pick a relative Euclidean thickening $e'$ of $(K,L)$ and denote by
  $e$ its restriction to $L$ \cite[thm.~A]{Klein_rel-thick}. In particular, $\cal T_\infty(K,L \rel e)$  is non-empty with basepoint defined by $e'$. 

Let $\xi$ represent an element $[\xi] \in [K,BG \rel \nu_\alpha]$. 
 Then by \eqref{eqn:thick-to-sphere}, stabilization by $-\xi$ defines a map
\begin{equation} \label{eqn:shift}
\sigma^{-\xi}\:   \cal T_\infty(K, L \rel e) \to  \cal T_\infty(K, L \rel \alpha)\, .
 \end{equation}
Hence, $ \cal T_\infty(K, L \rel \alpha)$ is also non-empty.

Note that
stabilization by stable spherical fibrations defines an action
\[
[K/L,BG]_\ast \times  \cal T_\infty(K, L \rel e) \to  \cal T_\infty(K, L \rel e)\, .
\]
(compare \cite[\S12]{klein2025poincaresurgery}).

Choose a stable trivialization of the Spivak normal
fibration of $e'$, and equip its restriction $e$ with the induced trivialization.
Define maps
\[
 [K/L,BG]_\ast @> s >>  \cal T_\infty(K, L \rel e) @>r >>  [K/L,BG]_\ast\, ,
 \]
as follows: $s$ is the map defined by the orbit of the relative Euclidean thickening and
$r :=r_e$ assigns to a relative thickening  its Spivak normal fibration. For $([\xi],x)  \in [K/L,BG]_\ast \times  \cal T_\infty(K, L \rel e)$,
we have the identity
\begin{equation} \label{eqn:identity-fibration}
r(\sigma^\xi(x)) = r(x) - [\xi]\, 
\end{equation}
by Lemma \ref{lem:whitney-sum-formula}. In particular, 
$r\circ s([\xi]) = -[\xi]$ and we see that $r$ is surjective. 

We next show that $r$ is injective.  If $r(x) = 0$, then
a representative of $x$ admits a stable normal trivialization extending the prescribed
trivialization on the common boundary face. By \cite[prop. 4.1]{Klein_rel-thick},
it follows that $x= e'$ in  $\mathcal T_\infty(K,L\rel e)$.

 Suppose that $r(x) = [\eta] = r(y) $ for some $x,y\in \cal T_\infty(K, L \rel e)$. Then by \eqref{eqn:identity-fibration}
 \[
 r(\sigma^{\eta}(x)) = 0 =  r(\sigma^{\eta}(y))\, .
 \]
It follows that $\sigma^{\eta}(x) = e' = \sigma^{\eta}(y)$. Now apply $\sigma^{-\eta}$ and the composition law
for twisted stabilization and use \eqref{eqn:sum-stable}
to obtain $x = y$. We conclude that $r$ is injective.
This proves the case $\alpha = e$. The general case follows by applying the shift in
\eqref{eqn:shift}.
\end{proof}

\begin{rem} The proof also shows that the map $s$ is a bijection. Consequently,
the action 
\[
[K/L,BG]_\ast \times \cal T_\infty(K, L \rel e) \to \cal T_\infty(K, L \rel e)
\]
is free and transitive.
\end{rem}

\begin{rem} Proposition \ref{prop:relL} implies that there is an exact sequence of sets
\[
\cal T_\infty(K,L \rel \alpha)  \to \cal T_\infty(K) \to \cal T_\infty(L ) 
\]
where in this case,  the basepoint of $\cal T_\infty(L )$ is represented by $\alpha$.
Here, $\cal T_\infty(K) \to \cal T_\infty(L )$ corresponds to 
the restriction $[K,BG] \to [L,BG]$
 (compare \cite[\S6]{Wall_thickening}).
\end{rem}

\subsection{Alternative description} \label{subsec:alternative-thickening} The following variant
of the notion of relative thickening
is used in the proof of Theorem \ref{bigthm:Smale-Hirsch}.

Let $(K,L)$ be a cofibrant pair as above, and let $(L,A_L)$ be a Poincar\'e pair of dimension $d-1$
with the property that the triad $(L;\emptyset, A_L)$ has homotopy spine dimension $\le d-4$.
Then 
\begin{equation} \label{eqn:alphaL}
\alpha_L := (L,A_L, \text{id}_L)
\end{equation}
is a Poincar\'e thickening of $L$  of dimension $d-1$.

Let $t_L\: A_L \to L$ be the inclusion.
Consider a map of pairs
\[
t := (t_K,t_L) \: (A_K,A_L) \to (K,L)\, ,
\]
where $t_K$ is not necessarily an inclusion.
Then $t$
may be considered as an object of 
$T(A_L \to K)$, the category whose objects are spaces
$X$ equipped with a factorization $A_L \to X \to K$, where a morphism
is a map of spaces that is compatible with the structure maps.
A morphism is a weak equivalence if and only if it is a weak homotopy equivalence
of topological spaces.

Let
$\bar K$ be the mapping cylinder of the map $L \cup_{A_L} A_K \to K$.
In this instance we declare that the map $t$ defines a Poincar\'e thickening of $(K,L)$  if
the triad 
\[
(\bar K; L, A_K)
\]
is a Poincar\'e triad of dimension $d$ having homotopy spine dimension $\le d-3$.

 Two
such thickenings are deemed {\it concordant} when they are isomorphic in the homotopy
category of  $T(A_L \to K)$.  Let 
\[
\cal T'_d(K,L\rel t_L)
\] denote the set of equivalence classes.

The two definitions of relative thickening agree:
\[
\cal T_d'(K,L \rel t_L)\cong \cal T_d(K,L \rel \alpha_L) \, ,
\]
where a bijection is induced by the operation $t\mapsto (\bar K; L, A_K)$; an equivalence
$(K,L) @>>> (\bar K,L)$ is defined by a relative homotopy inverse to the mapping cylinder collapse.


\end{document}